\documentclass{article}
\usepackage[utf8]{inputenc}

\usepackage{a4wide}

\usepackage[dvipsnames]{xcolor}
\usepackage{latexsym,amssymb}
\usepackage{amsmath,amsthm,amsxtra,amsbsy,accents}
\usepackage{mathtools}
\usepackage[scr=boondoxo]{mathalpha}
\usepackage{bm}
\usepackage{hyperref}

\usepackage{enumerate}
\usepackage{enumitem}

\usepackage{upgreek}

\DeclareMathAlphabet{\mathup}{OT1}{\familydefault}{m}{n}
\newcommand{\dx}[1]{\mathop{}\!\mathup{d} #1}

\newcommand{\ve}{\varepsilon}

\DeclarePairedDelimiter{\abs}{\lvert}{\rvert}
\DeclarePairedDelimiter{\norm}{\lVert}{\rVert}
\DeclarePairedDelimiter{\bra}{(}{)}

\DeclarePairedDelimiter{\set}{\{}{\}}
\DeclarePairedDelimiter{\skp}{\langle}{\rangle}

\numberwithin{figure}{section}

\def\calA{{\mathcal A}}  
  
 \def\calH{{\mathcal H}} \def\calI{{\mathcal I}}
\def\calJ{{\mathcal J}}  
\def\calM{{\mathcal M}}  
\def\calP{{\mathcal P}}  
\def\calS{{\mathcal S}} \def\calT{{\mathcal T}}

 \def\sfe{{\mathsf e}} \def\sff{{\mathsf f}}

\def\sfm{{\mathsf m}}  
  
\def\sfs{{\mathsf s}}  
\def\sfv{{\mathsf v}} \def\sfw{{\mathsf w}}

 \def\sfE{{\mathsf E}} 
\def\sfG{{\mathsf G}}  
  \def\sfL{{\mathsf L}}

\def\sfV{{\mathsf V}}

\def\scrA{{\mathscr  A}}

  \def\scrL{{\mathscr  L}}

\newcommand{\R}{\mathbb{R}}
\newcommand{\N}{\mathbb{N}}
\newcommand{\CE}{{\mathsf{CE}}}

\newcommand{\prox}{\text{prox}}

\newcommand{\supp}{{\rm supp}}

\newcommand{\leb}{\scrL}
\newcommand{\setsep}{:\,}
\newcommand{\restrto}[1]{_{|#1}}
\newcommand{\entsep}{\,|\,}

\newcommand{\Meas}{\calM}
\newcommand{\MeasPos}{\calM_{\geq 0}}
\newcommand{\Prob}{\calP}

\newcommand{\one}{\mathbf{1}}
\newcommand{\omge}{\sfL}
\newcommand{\omgg}{\sfG}
\newcommand{\ceunreg}{\CE(\mu_0,\mu_1)}
\newcommand{\cebound}{\CE_{b}(\mu_0,\mu_1)}
\newcommand{\cereg}{\CE_{\rm reg}(\mu_0,\mu_1)}
\newcommand{\ceregmix}{\CE_{\text{reg}}^{\rm mix}(\rho_0,\rho_1)}
\newcommand{\weakcvg}{\rightharpoonup}

\newcommand{\spaceflux}{L^1\bra*{(0,1)\times\omge}}

\newcommand{\Lip}{\mathrm{Lip}}

\newcommand{\diam}{{\rm diam}}

\makeatletter
\newcommand{\bnabla}{{\mathpalette\b@nabla\relax}}
\newcommand\b@nabla[2]{%
        \setbox\z@=\hbox{$\m@th#1\bigtriangledown$}%
        \ht\z@.7\ht\z@
        \raise\dp\z@\box\z@
}
\makeatother

\newtheorem{theorem}{Theorem}[section]
\newtheorem{assumption}[theorem]{Assumption}
\newtheorem{lemma}[theorem]{Lemma}

\newtheorem{definition}[theorem]{Definition}

\newtheorem{proposition}[theorem]{Proposition}
\newtheorem{remark}[theorem]{Remark}
\newtheorem{example}[theorem]{Example}

\numberwithin{equation}{section}

\makeatletter
\newcommand{\oset}[3][0ex]{%
  \mathrel{\mathop{#3}\limits^{
    \vbox to#1{\kern-2\ex@
    \hbox{$\scriptstyle#2$}\vss}}}}
\makeatother

\usepackage[font=small]{caption}
\usepackage{subcaption}
\usepackage{graphicx}
\usepackage{authblk}
\usepackage{faktor}

\usepackage{xcolor}
\usepackage[normalem]{ulem}

\DeclareMathOperator*{\argmin}{\arg\!\min}

\usepackage{tikz}
\usetikzlibrary{arrows.meta}
\usetikzlibrary{calc,intersections,through}
\usetikzlibrary{positioning}
\usetikzlibrary{bending}

\tikzset{
    >=Stealth, 
    every picture/.style={
        line width=0.75pt, 
        shorten >=1pt,     
        shorten <=1pt
    }
}

\usepackage{centernot}

\usepackage[vlined, ruled]{algorithm2e}

\definecolor{bluegray}{rgb}{0.4, 0.6, 0.8}
\definecolor{mntf}{HTML}{13574d}
\definecolor{unia}{HTML}{3c2673}
\definecolor{light_green}{HTML}{009b5d}

\title{The Schrödinger problem on metric graphs}
\author[1]{Juliane Krautz}
\author[1,2]{Jan-F. Pietschmann}
\affil[1]{{
\small
Universit\"{a}t Augsburg, Institut f\"ur Mathematik, Universit\"{a}tsstra\ss e 12a, 86159 Augsburg, Germany. Emails: \{juliane.krautz, jan-f.pietschmann\}@uni-a.de
}}
\affil[2]{{
\small
Centre for Advanced Analytics and Predictive Sciences (CAAPS), University of Augsburg,
Universit\"{a}tsstr. 12a, 86159 Augsburg, Germany. 
}}
\date{\today}

\begin{document}

\maketitle

\begin{abstract}
We study the Schrödinger problem on metric graphs and its different formulations. Starting from a static version, we introduce an equivalent reformulation as entropic optimal transport and show $\Gamma$-convergence towards static optimal transport. 
We then rigorously derive a Benamou-Brenier type dynamic version of the Schrödinger problem, thereby extending known results from ${\rm RCD}^*(K,N)$-spaces. With this equivalence at hand, we conclude that the minimum values of the dynamic Schrödinger problem converge towards the squared Wasserstein distance, and minimizers converge to Wasserstein geodesics. We also extend the dynamic formulation to a more general class of initial and final data and show existence of solutions in this setting using the direct method. 
Lastly, we illustrate our analytical findings by a numerical investigation. 
\end{abstract}


\tableofcontents

\section{Introduction}

In 1931, E. Schrödinger introduced a mathematical problem concerned with finding the most likely evolution between two observations $\mu_0$ and $\mu_1$ of distributions of independent gas particles \cite{Schroeinger31}. This thought experiment amounts to the so-called Schrödinger problem. 
In modern terms, initial and final observations are modeled by probability measures $\mu_0,\mu_1\in\Prob(X)$ on a Polish space $(X,d)$. For $\pi^0,\pi^1:X\times X \to X$ denoting the canonical projections, we define the set of \emph{admissible transport plans}
\begin{align*}
    \Gamma(\mu_0,\mu_1) \coloneqq \set*{\gamma\in\Prob(X\times X) \setsep \gamma(A\times X) = \mu_0(A), \, \gamma(X\times B) = \mu_1(B)},
\end{align*}
which encode the displacement of the mass. Up to a sign, the likelihood of such a displacement is described in terms of the \emph{Boltzmann entropy}
\begin{align}\label{eq:static_entropy}
\calH(\gamma\entsep R) := \begin{cases}
    \iint_{X\times X} \log\frac{\dx{\gamma}}{\dx{R}} \dx{\gamma} &: \gamma\ll R \\
    +\infty &: \text{else}
\end{cases},
\end{align}
where we fix a reference measure $R\in \MeasPos(X\times X)$.
Then, the Schrödinger problem reads
\begin{align}\label{eq:static_metric_SBP}
    \inf\limits_{\gamma\in\Gamma(\mu_0,\mu_1)} \calH(\gamma\entsep R).
\end{align}
In the special case of the heat kernel as a reference measure, this interpolation problem is closely related to the following Kantorovich formulation of optimal transport
\begin{align}\label{eq:static_metric_OT}
    \inf\limits_{\gamma\in\Gamma(\mu_0,\mu_1)} \iint_{X\times X} \frac{1}{2}d^2(x,y) \dx{\gamma(x,y)}.
\end{align}
Additionally,~\eqref{eq:static_metric_OT} defines a metric on the space of probability measures with finite second moment $\Prob_2(X)$ which is the so-called \emph{Wasserstein distance}
\begin{align*}
    W_2^2(\mu_0,\mu_1) \coloneqq \inf\limits_{\gamma\in\Gamma(\mu_0,\mu_1)} \iint_{X\times X} d^2(x,y) \dx{\gamma(x,y)}.
\end{align*}
With this distance function, $(\Prob_2(X), W_2)$ is a metric space \cite{Ambrosio_GF}.
To see the relation between both problems, let us consider the case $X=\R^d$, and~\eqref{eq:static_metric_SBP} with the choice $R_\beta = k_{\beta/2}\leb\otimes\leb$, where 
\begin{align*}
    k_t(x,y) = (4\pi t)^{-d/2} e^{-\abs{x-y}^2/4t}
\end{align*}
is the heat kernel, $\beta>0$, and $\leb\in\MeasPos(\R^d)$ denotes the Lebesgue measure. Inserting these choices into the entropy functional defined in~\eqref{eq:static_entropy} yields 
\begin{align}\label{eq:EntrOT_metric}
    \beta \calH(\gamma\entsep R_\beta) = \iint_{\R^d\times \R^d} \frac{1}{2}\abs{x-y}^2 \dx{\gamma(x,y)} + \beta \calH(\gamma\entsep\leb\times\leb) + \frac{\beta}{2}\log(2\pi\beta).
\end{align}
In particular, the Schrödinger problem leads to a perturbed transport problem, where the additional penalty is given by the entropy. Such formulations are called entropic optimal transport. Formally,~\eqref{eq:EntrOT_metric} converges to~\eqref{eq:static_metric_OT} as $\beta\to0$ and this formal connection has been made rigorous following two different arguments. One due to \cite{leonardsurvey}, proving $\Gamma$-convergence for the static formulations directly on spaces where a Large Deviation principle holds true. This additional assumption guarantees that the logarithm of the heat kernel asymptotically behaves like the squared distance.
Another approach is as a direct consequence of~\eqref{eq:EntrOT_metric} and the $\Gamma$-convergence between entropic optimal transport and the Kantorovich formulation on $\R^d$ proven in \cite{carlier2017}.

In the past decades, optimal transport has received substantial interest, and in the seminal paper \cite{Benamou2000ACF} a fluid-dynamics formulation of~\eqref{eq:static_metric_OT} has been introduced. This dynamic problem reads as
\begin{align}\label{eq:dynamic_metric_OT}
    \inf\limits_{(\mu_t,v_t\mu_t)} \int_0^1 \int_X \frac{1}{2}\abs{v_t}^2\dx{\mu_t}\dx{t},
\end{align}
where the infimum runs over weak solutions to the continuity equation 
\begin{align*}
    \begin{cases}\partial_t \mu_t + \nabla\cdot (v_t \mu_t) =0 \\ \mu\restrto{t=0} = \mu_0, \quad \mu\restrto{t=1} = \mu_1 \end{cases}
\end{align*}
with prescribed initial and final data. Similarly, the Schrödinger problem admits a dynamic formulation as well, see \cite{leonardsurvey, Chen2016}. It is given by
\begin{align}\label{eq:dynamic_metric_SBP}
    \inf\limits_{(\mu_t,v_t\mu_t)} \int_0^1 \int_X \frac{1}{2}\abs{v_t}^2 \dx{\mu_t}\dx{t} + \frac{\beta^2}{8}\int_0^1 \int_X \abs{\nabla \log \mu_t}^2 \dx{\mu_t}\dx{t},
\end{align}
again minimized over weak solutions to the continuity equation. 
Comparing both~\eqref{eq:dynamic_metric_SBP} and~\eqref{eq:dynamic_metric_OT}, they differ only by an additional term appearing in the Schrödinger problem, called \emph{Fisher information}. In general, this term leads to higher regularity of the solutions as it requires at least weak differentiability to be finite. Therefore, the dynamic Schrödinger problem can also be understood as a regularized approximation of dynamic optimal transport.
On ${\rm RCD}^*(K,N)$-spaces, \cite{gigli_benamou-brenier_2018} showed that static and dynamic Schrödinger problem are equivalent in analogy to the celebrated Benamou-Brenier formulation of optimal transport, thus extending previously known results to a more general class of spaces. Very recently, \cite{garatti_pde_nodate} developed a new approach in $\R^d$ which is based on a PDE-perspective not relying on the ${\rm RCD}^*(K,N)$-structure. 

Due to this equivalence, $\Gamma$-convergence is expected to hold true for the dynamic problems as well. Indeed, in the case of ${\rm RCD}(K,\infty)$-spaces it has been shown in \cite{monsaingeon2023dyn}. As a consequence, the relations between the different problems and their reformulations can be summarized as in Figure~\ref{fig:connect_diff_problems}.  

\begin{figure}
    \centering
    \begin{tikzpicture}[
    formula/.style={
        draw,
        rounded corners,
        inner sep=5pt,
        align=center,
        font=\scriptsize
    }
    ]
    
    \node[formula] (n1) at (-4,2) {$\inf_\gamma \iint \frac{1}{2} d^2(x,y) \dx{\gamma(x,y)} + \beta\calH(\gamma\entsep\sfm\otimes\sfm) $};
    \node[formula] (n2) at (0,0) {$\inf_\gamma \beta\calH(\gamma\entsep R_\beta)$};
    \node[formula] (n3) at (4,2) {$\inf_{(\mu_t,v_t\mu_t)} \int_0^1\int \bra*{\frac{1}{2}\abs{v_t}^2 + \frac{\beta^2}{8} \abs{\nabla\log\mu_t}^2 }\dx{\mu_t } \dx{t}$};
    
    \node[formula] (n4) at (-4,-2) {$\inf_\gamma \iint \frac{1}{2}d^2(x,y) \dx{\gamma(x,y)}$};
    \node[formula] (n5) at (4,-2) {$\inf_{(\mu_t,v_t\mu_t)} \int_0^1\int \frac{1}{2}\abs{v_t}^2 \mu_t \dx{\mu_t} \dx{t}$};
    
    \draw[<->, double, light_green]
        (n1) -- node[above, sloped, font=\tiny] {$\rm SP \leftrightarrow EOT$} (n2);
    
    \draw[<->, double, light_green]
        (n2) -- node[above, sloped, font=\tiny] {$\rm SP \leftrightarrow D_{SP}$} (n3);
    
    \draw[<->, double]
        (n4) -- node[below, sloped, font=\tiny] {$\rm OT \leftrightarrow D_{OT}$} (n5);
    
    \draw[<-, thick, light_green]
        (n4) -- node[above, sloped, font=\tiny] {$\rm OT \xleftarrow{\Gamma} EOT$} (n1);
    
    \draw[->]
        (n3) -- node[above, sloped, font=\tiny] {$\rm D_{SP} \xrightarrow{\Gamma} D_{OT}$} (n5);

    \draw[<-]
    (n2) -- node[above, sloped, font=\tiny] {$\rm OT \xleftarrow{\Gamma} SP$} (n4);
    
    \end{tikzpicture}
    \caption{Relation between optimal transport and Schrödinger problems on an ${\rm RCD}^*(K,N)$-space with reference measure $\sfm$. On metric graphs, the entropic transport problem is given in terms of a slightly modified cost function and the dynamic Schrödinger problem minimizes over bounded curves instead. In this article, the highlighted relations are proven to hold on metric graphs.}
    \label{fig:connect_diff_problems}
\end{figure}
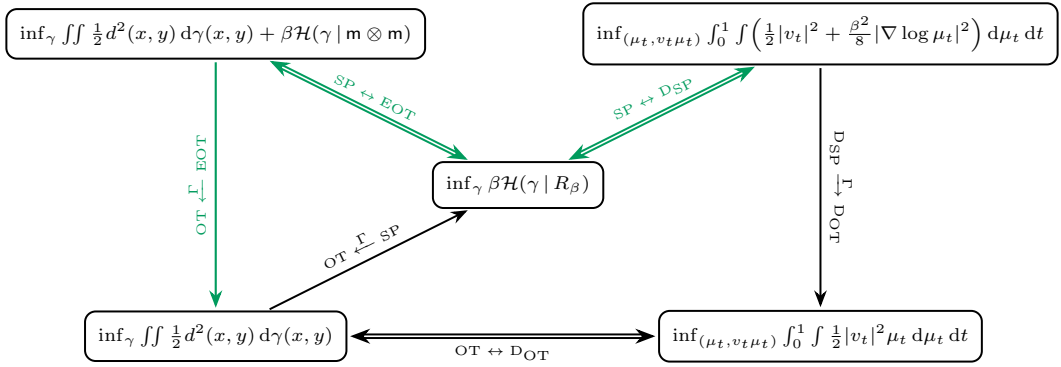

To the best of the authors' knowledge, the full picture of Figure~\ref{fig:connect_diff_problems} is only available on ${\rm RCD}^*(K,N)$-spaces so far, and even partial results are unknown outside this setting. A prototypical example of a metric space not contained in this class are metric graphs. On these spaces, static and dynamic optimal transport has been studied in \cite{MatthiasErbar2022} (see also \cite{Burger2023,FazenyBurgerPietschmann2026} for the case allowing mass storage on vertices), establishing the equivalence $\rm OT \leftrightarrow D_{OT}$ from Figure~\ref{fig:connect_diff_problems}. Moreover, a rich theory for the heat equation is available, among which is an explicit formula for the heat kernel derived in \cite{roth_spectre_1984}. Further, spectral properties \cite{kurasov_spectral_2024}, Gaussian estimates \cite{mugnolo_gaussian_2007, baudoin_differential_2018}, and semigroup properties \cite{fijavz_variational_2007, mugnolo_semigroup_2014} have been studied. Relying on this theory, we analyze the Schrödinger problem on metric graphs, focusing on its relation to optimal transport.  

\textbf{Aims and contributions.} The aim of this paper is to study both dynamic and static formulations of the Schrödinger problem on metric graphs and to investigate their connection to optimal transport. In particular, we show that the relations $\rm SP \leftrightarrow EOT$, $\rm EOT \overset{\Gamma}{\to} OT$, and $\rm SP \leftrightarrow D_{SP}$ highlighted in Figure~\ref{fig:connect_diff_problems} hold true on metric graphs with a modified cost function in the entropic optimal transport formulation. Interestingly, we do not obtain $\Gamma$-convergence in terms of the dynamic formulations. This is mainly due to the lack of stronger lower curvature bounds on metric graphs. However, the arising difficulties can be circumvented in the static point of view, allowing to conclude $\Gamma$-convergence there. A more detailed discussion of this phenomenon is given in Subsection~\ref{sec:discussion}.

\textbf{Structure of the paper.} 
First, we introduce the necessary notation as well as properties of the heat kernel on metric graphs in Section~\ref{Prelim}. This preliminary section is concluded with a discussion of synthetic curvature conditions on metric graphs and their connection to the Schrödinger problem, focusing on the differences between static and dynamic formulations. 
In Section~\ref{sec:static}, we study the static problems, prove existence of solutions for each of them, and give a reformulation of the Schrödinger problem as an entropic transport one. Additionally, this section contains our first main result, namely $\Gamma$-convergence between entropic and Kantorovich optimal transport.
In Section~\ref{sec:dynamic}, we prove our second main result which is the dynamic formulation of the Schrödinger problem. Moreover, we show uniqueness of its minimizers, and as a consequence, we give a convergence result for the minimum values and the minimizers in the dynamic setting. 
We also extend the previously found dynamic problem to more general initial and final data and study existence of solutions in this case. 
Lastly, a numerical method to solve the dynamic Schrödinger problem based on augmented lagrangian formulations and primal dual methods is proposed and tested for several examples in Section~\ref{numerics}.

\section{Preliminaries}\label{Prelim}

We first introduce our notation and collect preliminary results on the heat kernel on metric graphs and synthetic curvature conditions on these spaces.

\subsection{Notation}

Let $\sfG = (\sfV, \sfE)$ be an undirected combinatorial graph with a finite number of vertices $\sfV$ and edges $\sfE\subset\sfV\times\sfV$. Throughout, we assume the graph to be connected without self-loops, meaning that any pair of vertices can be connected by a path and that there are no edges $\sfe\in\sfE$ of the form $\sfe=(\sfv,\sfv)$ for some $\sfv\in\sfV$.
We endow the graph with a map $\ell:\sfE \to (0,\infty)$,~$\ell(\sfe) =: \ell_\sfe$ associating a positive length to the edges and define an orientation on each edge using the outer normal 
\begin{align*}
    n^\sfe_\sfv := \begin{cases}
    -1 &: \sfe = (\sfv,\sfw) \\
    +1 &: \sfe = (\sfw,\sfv) \\
    0 &: \sfv\notin \sfe
\end{cases},
\end{align*}
thus allowing us to identify edges with intervals $[0,\ell_\sfe]$. We call the triple $\sfG:=(\sfV, \sfE, \ell)$ a \emph{metric graph}.
For each $\sfv\in\sfV$ we further denote the set of \emph{incident edges} by $\sfE(\sfv) := \{\sfe\in\sfE \ | \sfv\in \sfe\}$.

In order to define functions on metric graphs, let
\begin{align*}
\omge := \bigsqcup\limits_{\sfe\in\sfE} [0,\ell_\sfe] \quad \text{and} \quad \omgg := \faktor{\omge}{\sim},
\end{align*}
where $\sim$ denotes the equivalence relation that identifies endpoints of the intervals with vertices. 
Note that vertices are not uniquely determined in the set $\omge$ as it contains a copy for each incident edge, while they are in $\omgg$ due to the identification.
We denote by $\leb$ the measure on $\omgg$ that is obtained from the Lebesgue measure on $\omge$ under the equivalence relation $\sim$.
Functions on metric graphs are defined as maps on either $\omge$ or $\omgg$ and given by their restriction to each edge. We write $f:\Omega \to \R$ with $f=\bra{f^\sfe}_{\sfe\in\sfE}$ and denote by $L^1(\Omega)$ the corresponding set of Lebesgue-measurable functions for $\Omega\in\set{\omge, \omgg}$ with norms
\begin{align*}
    \norm{f}_{L^1(\omgg)} := \int_{\omgg} \abs{f}(x) \dx{x}~,\quad \norm{f}_{L^1(\omge)} := \int_{\omge} \abs{f}(x) \dx{x} = \sum\limits_{\sfe\in\sfE} \int_{[0,\ell_\sfe]} \abs{f^\sfe(x)} \dx{x}.
\end{align*}
The resulting Sobolev space is given by $W^{1,1}(\omge) := \set*{f\in L^1(\omgg) \setsep f^\sfe \in W^{1,1}\bra*{(0,\ell_\sfe)}}$. We write $\supp(f) \coloneqq\set{x\in\omgg\setsep f(x) >0}$ for the support of the function $f:\omgg\to\R$.
On the set $\omgg$, the graph distance $d_\sfG:\omgg\times\omgg \to [0,\infty)$ defines a metric, making $(\omgg, d_\sfG)$ a geodesic metric space \cite{MatthiasErbar2022}. In this metric space, we denote the set of continuous functions by $C(\omgg)$ and define
\begin{align*}
    C^1(\omgg) := \set*{f \in C(\omgg) \setsep f^\sfe\in C^1\bra*{[0,\ell_\sfe]} \text{ for all } \sfe\in\sfE}.
\end{align*}
The set of \emph{Lipschitz functions} is given as $C^{0,1}(\omgg) := \set*{f\in C(\omgg) \setsep \Lip(f) <\infty}$, where 
\begin{align*}
    \Lip(f) &:= \sup\limits_{x\neq y} \frac{\abs{f(x) - f(y)}}{d_\sfG(x,y)}
\end{align*}
is the (global) \emph{Lipschitz constant} on the metric graph $\sfG$. Further, we introduce the notation 
\begin{align}\label{eq:Notation_C1Cinfty}
    C^1((0,1))\cap C^\infty(\omge) \coloneqq \set*{f_t:(0,1)\times\omge \to \R \setsep \genfrac{}{}{0pt}{}{t\mapsto f_t(x)\in C^1((0,1)) \text{ for all } x\in\omge,} {\, x\mapsto f_t(x) \in C^\infty(\omge) \text{ for all } t>0}}.
\end{align}

For $\Omega\in\set{\omge,\omgg}$ we define $\calM(\Omega)$ as the set of \emph{Borel measures} and $\calM_{\geq0}(\Omega)$ as the set of \emph{non-negative Borel measures}. The space of \emph{probability measures} is denoted by $\calP(\Omega)$ and the total variation norm of any measure $\nu\in\calM(\Omega)$ by $\norm{\nu}$. 
Let 
\begin{align}\label{eq:static_OT}\tag{OT}
    W_2^2(\mu_0, \mu_1) = \inf\limits_{\gamma \in \Gamma(\mu_0,\mu_1)} \iint_{\omgg\times\omgg} d_\sfG(x,y)^2 \dx{\gamma(x,y)}
\end{align}
denote the 2-Wasserstein distance on a metric graph. 
As noted in \cite{MatthiasErbar2022}, $\calP(\omgg)$ endowed with this distance is a geodesic space and $W_2(\mu_0,\mu_1)<\infty$ for all $\mu_0,\mu_1\in\Prob(\omgg)$.
It is well known that the $2$-Wasserstein distance metricizes the weak convergence of probability measures on compact metric spaces, see e.g. \cite[Proposition 7.1.5] {Ambrosio_GF}. 
For two measures $\mu\in\calM(\Omega)$ and $\sfm\in\calM_{\geq0}(\Omega)$, $\Omega\in\set{\omge, \omgg}$, we write $\mu\ll \sfm$ if $\mu$ is absolutely continuous with respect to $\sfm$ and denote the Radon-Nikodym densities of $\mu$ by $\frac{\dx{\mu}}{\dx{\sfm}}$. If $\mu$ is singular with respect to $\sfm$, we write $\mu\perp\sfm$.
For any measure $\nu\in\calM([0,1]\times\Omega)$, $\Omega\in\set*{\omgg, \omge}$, we write $\nu = (\nu_t)_{t\in[0,1]}$ for its disintegration if it is well-defined. On the other hand, any family of measures $(\nu_t)_{t\in[0,1]}$, $\nu_t\in\calM(\Omega)$ for $t\in[0,1]$, defines a measure $\nu \in \calM([0,1]\times\Omega)$ by its disintegration.

We want to highlight that functions in $C(\omgg)$ are continuous up to the boundary on all edges and that $\omgg$ and $\omge$ are compact. Therefore, the duality between Borel-measures holds with respect to this set of functions.

As in~\eqref{eq:static_entropy}, we denote the \emph{Boltzmann entropy} on a metric graph by $\calH(\gamma\entsep R)$ for $\gamma\in\Prob(\omgg\times\omgg)$ and a reference measure $R\in \MeasPos(\omgg\times\omgg)$. By slight abuse of notation, we also write 
\begin{align*}
    \calH(\mu\entsep \leb) := \begin{cases}
    \int_{\omgg} \log\frac{\dx{\mu}}{\dx{\leb}} \dx{\mu} &: \mu\ll \leb \\
    +\infty &: \text{else}
\end{cases}
\end{align*}
for the entropy of a probability measure $\mu\in\Prob(\omgg)$.
If $\mu\in \Prob(\omgg)$ is given such that $\calH(\mu\entsep\leb)<+\infty$, then $\mu\ll\leb$ as a direct consequence of the definition and by the theorem of Radon-Nikodym, there exists a density $\rho = \frac{\dx{\mu}}{\dx{\leb}}$.

As a last important notion, we recall the definition of $\Gamma$-convergence. 

\begin{definition}[$\Gamma$-convergence]\label{def:Gammcvg}
    Let $X$ be a topological space with a sequence of functionals $F_\beta:X\to \R\cup\{\infty\}$ for $\beta>0$. A functional $F:X\to\R\cup\{\infty\}$ is the \emph{$\Gamma$-limit} of this sequence if 
    \begin{enumerate}[label=\roman*)]
        \item for all $x_\beta\to x$ in $X$ as $\beta\to 0$ it holds that $F(x) \leq \liminf_{\beta\to 0} F_\beta(x_\beta)$ and
        \item for all $x\in X$ there exists a \emph{recovery sequence} $\bra*{x_\beta}_{\beta>0}\subset X$ such that $x_\beta\to x$ as $\beta\to 0$ and $F(x) = \lim_{\beta\to 0} F_\beta(x_\beta)$. \label{RecovSeq}
    \end{enumerate}
\end{definition}

\begin{remark}
    If the liminf-condition holds, it suffices to show the existence of a sequence $x_\beta\to x$ such that $F(x) \geq \limsup_{\beta\to 0} F_\beta(x_\beta)$ to conclude~\ref{RecovSeq}.
\end{remark}

\subsection{The heat kernel on metric graphs}

In what follows, we fix a specific kind of reference measure in the static Schrödinger problem that is related to the heat kernel. For this choice, we study equivalent formulations as well as the $\Gamma$-limit.
To this end, we rely on regularity properties of the heat semigroup $(H_t)_{t\geq 0}$ on a metric graph and its kernel. The semigroup itself is defined via solutions to the heat equation under so-called standard coupling conditions, i.e. 
\begin{align}\label{eq:heat}
    \begin{cases}
    \partial_tH_t\mu = \Delta H_t\mu &\text{in } (0,\infty)\times\omge\\
    \sum_{e\in\sfE(\sfv)} \partial_x H_t\mu^\sfe \cdot n^\sfe_\sfv = 0 &\text{for } t>0, \, \sfv\in\sfV\\
    H_t\mu^\sfe(\sfv) = H_t\mu^{\sff}(\sfv) &\text{for } t>0, \, \sfv\in\sfV\\
    H_0\mu = \mu &\text{in } \omgg
\end{cases}.
\end{align}
As was shown in \cite{roth_spectre_1984}, this semigroup gives rise to a heat kernel $h_t(x,y)$. 
In order to characterize the kernel in more detail, we need to introduce additional notation first. 
We denote by $\deg(\sfv) := \abs{\sfE(\sfv)}$ the \emph{degree} of $\sfv\in\sfV$ and identify each edge $\sfe\in\sfE$ with two copies $\sfe_+$ and $\sfe_-$, called arcs, that only differ in orientation. Here, $\sfe_+$ is oriented according to $n^\sfe_\sfv$. The initial vertex on $\sfe_\pm$ is denoted by $I(\sfe_\pm)$, whereas $T(\sfe_\pm)$ denotes the terminal vertex. We write $\vec \sfe$ for any of the two arcs $\sfe_\pm$ resulting from the edge $\sfe\in\sfE$. Given $\vec \sfe$ we denote the arc with opposite orientation by $-\vec \sfe$.
For any pair of arcs $\vec \sfe,\vec \sff$ we define the \emph{transfer coefficient}
\begin{align*}
    \ve^\sfG_{\vec \sfe,\vec \sff} := \begin{cases}
        \frac{2}{\deg(T(\vec \sfe))} &: T(\vec \sfe)=I(\vec \sff) \text{ and } T(\vec \sff) \neq I(\vec \sfe)\\
        \frac{2}{\deg(T(\vec \sfe))} - 1 &: T(\vec \sfe)=I(\vec \sff) \text{ and } T(\vec \sff) = I(\vec \sfe)\\
        0 &: T(\vec \sfe)\neq I(\vec \sff)
    \end{cases}.
\end{align*}
Combinations of the first kind are called \emph{traverses} and combinations of the second kind \emph{reflections} at the vertex $T(\vec \sfe)=I(\vec \sff)$.
For coupling conditions different than the standard ones, these coefficients change and we refer to \cite{kostrykin_heat_2007} for a detailed analysis.

In the above formula, vertices of degree $2$ are negligible as traverses through them have transfer coefficient one, whereas reflections lead to the coefficient being zero. Therefore, we are able to identify any metric graph with a graph that has no degree-$2$ vertices. At the same time, any point $x\in\omgg$ can be identified with a vertex of degree $2$, changing the set of vertices accordingly but not the heat-kernel.
Any finite sequence of arcs $P=(x,\vec \sfe_1,\ldots,\vec \sfe_{k+2},y)$, $k\in\N_0$, with $I(\vec \sfe_1) = x$,~ $T(\vec \sfe_{k}) = y$ and $T(\vec \sfe_i)=I(\vec \sfe_{i+1})$ for all $i\in\{1,\ldots,k-1\}$, is called a \emph{path}. The first and last arc of such a sequence connect $x,y\in\omgg$ to a vertex $\sfv\in\sfV$. The length of these arcs is bounded from above by the maximum edge length. The remaining arcs $\vec \sfe_k$ in the sequence are of length $\ell_{\sfe_k}>0$, connecting two vertices with each other.
For $k\in\N_0$, we denote by $P_{k+2} (x,y)$ the set of paths starting in $x$ and ending in $y$ containing $k$ full edges and therefore having length 
\begin{align*}
    l(P) := \begin{cases}
        d_\sfG(x,y) &: k=0\\
        \sum\limits_{j=2}^{k+1} \ell_{\sfe_j} + d_\sfG(x,T(\vec \sfe_1)) + d_\sfG(y, I(\vec \sfe_{k+2})) &: k\geq 1
    \end{cases}.
\end{align*}
Each path $P=(x,\vec \sfe_1,\ldots,\vec \sfe_{k+2},y)$ gives rise to the transfer coefficient 
\begin{align*}
    \ve^\sfG_P := \prod\limits_{i=1}^{k+1} \ve^\sfG_{\vec \sfe_i, \vec \sfe_{i+1}}.
\end{align*}

With this notation, we are able to characterize the heat kernel for standard boundary conditions using the path-sum formula introduced in \cite{roth_spectre_1984}.

\begin{definition}\label{kernelPath}
    The heat kernel on a metric graph $\sfG$ is given by
    \begin{align*}
        h_t(x,y) := k_t(d_\sfG(x,y)) \delta_{\sfe,\sff} + L_t(x,y),
    \end{align*}
    where $\delta_{\sfe,\sff}$ is the Kronecker symbol,
    \begin{align*}
        L_t(x,y) := \sum\limits_{k=0}^\infty ~\sum\limits_{P\in P^\sfG_{k+2}(x,y)} \ve^\sfG_P k_t(l(P))
    \end{align*}
    for $x\in\sfe$, $y\in\sff$ and 
    \begin{align*}
        k_t(z) := \frac{1}{\sqrt{4\pi t}} e^{-\frac{z^2}{4t}}
    \end{align*}
    is the heat kernel on the real line.
\end{definition}

The above formula holds true on $\omgg$ as was shown in \cite{kostrykin_heat_2007}, and is independent of the choice of edges $\sfe\in \sfE(v)$ for $\sfv\in\sfV$. 
In addition to the path-sum formula from Definition~\ref{kernelPath} the heat kernel can be written by means of its eigenvalue expansion \cite{kurasov_spectral_2024} which reads
\begin{align}\label{eq:eigenexp}
    h_t(x,y) = \sum\limits_{k=1}^\infty e^{-\lambda_k t} \psi_k(x) \psi_k(y)
\end{align}
for eigenvalues $0= \lambda_1 \leq \lambda_2 \leq \ldots$ and eigenfunctions $\psi_k\in L^2(\omgg)$ forming an orthonormal basis of $L^2(\omgg)$. In the case of standard boundary conditions, $\lambda_1=0$ is an eigenvalue of the Laplacian with multiplicity one and $\psi_1\equiv \rm c$ for a constant $c>0$. 
Moreover, the heat kernel admits gaussian estimates.
The following is the content of \cite[Lemma 5.3]{baudoin_differential_2018}.

\begin{lemma}\label{lem:GaussianBnds}
    For given $T>0$, there exists $T_0\leq T$ and constants $C_0,C_1>0$ such that  
    \begin{align*}
        C_1 k_t(d_\sfG(x,y)) \geq h_t(x,y) \geq C_0 k_t(d_\sfG(x,y)) > 0
    \end{align*}
    for $t\in(0,T_0)$ and a.e. $x,y\in\omgg$. The constants $C_0$, $C_1$ and $T_0$ depend only on the geometry of the underlying metric graph and the time-bound $T>0$. 
\end{lemma}

The heat kernel indeed characterizes solutions to the heat equation and preserves integrability properties of the initial data \cite[Theorem 3.6]{fijavz_variational_2007}. The definition extends to measure valued initial data by duality. 
\begin{lemma}\label{lem:RegHeat}
    For $f\in L^p(\sfG)$ with $p\in [1,\infty)$ it holds that
    \begin{align*}
        H_t f (x) \coloneqq \int_{\omgg} h_t(x,y) f(y) \dx{y} \overset{t\searrow 0}{\longrightarrow} f \quad \text{in } L^p(\omgg) 
    \end{align*}
    and $H_tf(x)>0$ for $f\geq 0$ and not the constant zero function.
    Moreover, if $p>1$ then $H_tf \in C^\infty(\omge)$ for all $t>0$, and it solves the heat equation~\eqref{eq:heat} pointwise for $t>0$ and $x\in\omge$. For $f\in L^\infty(\omgg)$ we have $\abs{H_tf(x)}\leq \norm{f}_{L^\infty(\omgg)}$ for all $t>0$, $x\in\omgg$ and if $\mu\in\Prob(\omgg)$, then $H_t\mu(x) = \int_\omgg h_t(x,y) \dx{\mu(y)}$ defines a distributional solution to~\eqref{eq:heat}.
\end{lemma}

\subsection{Synthetic curvature conditions on metric graphs}\label{sec:discussion}

So far the full picture of Figure~\ref{fig:connect_diff_problems} is only known under synthetic lower curvature bounds in the sense of ${\rm RCD}(K,\infty)$-conditions. Without such bounds, only partial results are available, with proofs specifically tailored to fit the structure of the metric space. This is necessary as the general theory relies on particular properties of the heat flow, such as an $\rm EVI$-formulation or Wasserstein contractivity, which are closely linked to curvature. 
Metric graphs, on the other hand, provide standard examples where these conditions fail, which motivated the study of even weaker notions and it turns out that such weak lower curvature bounds are available \cite{krautz2025weak}. However, they do not suffice to adapt the arguments from the ${\rm RCD}(K,\infty)$-setting and conclude the $\Gamma$-limit. For this reason, another approach must be taken. 

To this end, we discuss the challenges that arise on metric graphs compared to ${\rm RCD}(K,\infty)$-spaces, focusing on their connection to lower curvature bounds. Additionally, we illustrate the differences under weaker curvature conditions and compare the static and dynamic formulations with respect to this question. Interestingly, even though they turn out to be equivalent (Theorem \ref{thm:equiv_static_dynamic}), both formulations behave differently with respect to curvature bounds. 
Motivated by the abstract theory, we study the relation between optimal transport and the Schrödinger problem in terms of $\Gamma$-convergence, see Definition~\ref{def:Gammcvg}, with respect to weak convergence in the space of measures.

We first discuss the dynamic formulation~\eqref{eq:dynamic_metric_SBP}. 
While the asymptotic lower bound follows from standard lower semi-continuity, the limitation lies in the asymptotic upper bound, more precisely in the construction of suitable recovery sequences. 
In our setting, any such sequence is required to satisfy two main properties. It has to be absolutely continuous with an action asymptotically bounded from above by the action of the unregularized curve, and it has to be weakly differentiable in order for the Fisher information to be finite. On metric graphs, the second property already implies continuity \cite[Lemma 3.27]{mugnolo_semigroup_2014}, thus imposing a strong regularity condition.
The absolute continuity property, on the other hand, can be achieved by constructing solutions to the continuity equation, as was shown in \cite{MatthiasErbar2022}. Regarding the asymptotic bounds, we begin by reviewing the approach in ${\rm RCD}(K,\infty)$-spaces which relies on the theory of Wasserstein gradient flows. 
From this point of view, the Fisher information admits another interpretation, namely as the dissipation of the relative entropy along the heat flow \cite{Ambrosio_GF}. Therefore, regularizing by means of this flow is a natural candidate for constructing recovery sequences. 
In order to obtain admissible curves, initial and final conditions must be preserved, and to this end, an additional drift term is introduced. The resulting ansatz has been analyzed in \cite{monsaingeon2023dyn}, and we sketch the ideas presented there to illustrate the key challenges on metric graphs.
Let $h(t) =\min\{t,1-t\}$ be a hat function, and let $(\mu_t,v_t\mu_t)$ solve the continuity equation with finite action $\int_0^1\int_{\omge}|v_t|^2\dx{\mu_t}\dx{t} < +\infty$. Then, for $\beta>0$ we define the candidate for the recovery sequence as follows
\begin{align*}
    \mu_t^\beta = H_{\beta h(t)} \mu_t.
\end{align*}
Note that this new curve depends on time in two ways. On the one hand due to the initial curve $t\mapsto\mu_t$, while on the other hand by the heat flow regularization $t\mapsto H_{\beta h(t)} \mu$. We write $\partial_s H_s\mu$ for the derivative with respect to the time-variable. Formally, from the chain rule and the definition of the heat flow, we obtain that
\begin{align}\label{eq:formal_dt}
    \frac{\dx{}}{\dx{t}} \mu_t^\beta &= \beta h'(t) \cdot \partial_s H_{s}\mu_t\big|_{s=\beta h(t)} + \frac{\dx{}}{\dx{t}} H_{s}\mu_t \big|_{s=\beta h(t)} = \beta h'(t) \cdot \Delta H_{\beta h(t)}\mu_t + \frac{\dx{}}{\dx{t}} H_{s}\mu_t \big|_{s=\beta h(t)} . 
\end{align}
After rearranging the terms, the left-hand side of this equation corresponds to the metric speed $|\smash{\Dot{\mu}_t^\beta}|$ of the regularized curve. The right-hand side contains two contributions; the first term can be modified to equal the Fisher information, and the last term will define a relation between the new action and the old one in the following way. On ${\rm RCD}(K,\infty)$-spaces, it has been shown \cite{ambrosio_bakryemery_2015} that the synthetic curvature bound can be characterized equivalently by the inequality 
\begin{align}\label{eq:W2_contr}
    W_2(H_{t}\mu, H_{t}\nu) \leq e^{-Kt} W_2(\mu,\nu)
\end{align}
for $\mu$, $\nu$ probability measures with finite second moment, $t>0$, and $K$ the lower curvature bound. This estimate is also called Wasserstein contractivity. For the curves $\mu_t^\beta$, it implies 
\begin{align*}
    \frac{W_2(H_{\beta h(t)}\mu_s, H_{\beta h(t)}\mu_{s+h})}{h} \leq e^{-K\beta h(t)} \frac{W_2(\mu_s, \mu_{s+h})}{h} 
\end{align*}
for all $h>0$ and $s\in (0,1)$, so that in the limit we obtain 
\begin{align*}
    \abs*{\bra*{\frac{\dx{}}{\dx{t}}{H_s\mu_t}}\Big|_{s=\beta h(t)}} \leq e^{-K \beta h(t)} \abs{\Dot{\mu_t}},
\end{align*}
thus recovering the action of the original curve. Moreover, an asymptotic bound on the Fisher information can be inferred from an ${\rm EVI}$-formulation of the heat flow together with the chain rule. Combining both, the bound on the action and the bound on the Fisher information, yields the desired upper bound in the limit.

On metric graphs, a weaker notion of curvature bounded from below has been introduced in \cite{krautz2025weak}, where several equivalent characterizations are given. While the study of these conditions is interesting in its own right, we focus on one particular consequence here, namely the weak counterpart of~\eqref{eq:W2_contr}. For a constant $C>1$ depending only on the geometry of the specific graph, it holds that
\begin{align*}
    W_2(H_{t}\mu, H_{t}\nu) \leq C e^{-Kt} W_2(\mu,\nu)
\end{align*}
and in analogy to the ${\rm RCD}(K,\infty)$-setting, this implies the (again weaker) lower bound $K>0$ on the curvature. Qualitatively, both statements are of the same type; they imply a Wasserstein control of solutions to the heat equation by their initial values, as well as similar long-time behaviour. For small times, however, the results have different implications due to the additional constant. 
Looking back at the construction of a possible recovery sequence $\mu_t^\beta$, the weaker condition gives  
\begin{align*}
    \abs*{\bra*{\frac{\dx{}}{\dx{t}}{H_s\mu_t}}\Big|_{s=\beta h(t)}}  \leq C e^{-K \beta h(t)} \abs{\Dot{\mu_t}},
\end{align*}
from which we recover the action of the initial curve, but only up to the factor $C>1$. Due to this additional constant, the resulting inequality does not suffice to conclude the asymptotic upper bound. 
At the same time, for any possible choice of recovery sequence, we need to ensure finiteness of the Fisher information for a particular weak solution to the continuity equation. This naturally leads to the heat equation and its short-time behaviour in the limit $\beta\to0$. 

In general, other regularizations of absolutely continuous curves are possible. One such method has been introduced in \cite{MatthiasErbar2022}.
It is based on an extension of the metric graph combined with averaging, and plays a key role in the development of the equivalence between static and dynamic optimal transport.
Following their method, for a given absolutely continuous curve $t\mapsto \mu_t$ and any $n\in\N$ we find a curve $t\mapsto \mu_t^n $ satisfying
\begin{align*}
    \limsup\limits_{n\to\infty} \abs{\smash{\Dot{\mu}_t^n}} \leq \abs{\Dot{\mu_t}}.
\end{align*}
In particular, we recover the action of the original curve without an additional constant. Unfortunately, their method does not generate sufficient regularity for the Fisher information to be finite, and is therefore not applicable to recovery sequences. 
So far, no approximation combining both properties is available on metric graphs.

Let us now turn to the static formulation instead. Known approaches to $\Gamma$-convergence in this setting are based on a reformulation as entropic optimal transport, e.g. \cite{leonardsurvey, carlier2017}. They rely on an expansion of the heat kernel in the form $\beta \log h^{-1}_{\beta/2}(x,y) = \frac{1}{2}d^2(x,y) + o(\beta)$, where $d(\cdot,\cdot)$ is the canonical distance on the metric space. 
Such an estimate holds true on metric graphs as well. 
In this case, the key ingredients are the Gaussian upper and lower bounds from Lemma~\ref{lem:GaussianBnds}, which can be shown independently of curvature conditions by means of semigroup techniques \cite{mugnolo_gaussian_2007} or by combinatorial arguments \cite{baudoin_differential_2018}. Both approaches rely on the explicit path sum formula from Definition~\ref{kernelPath}. 
Remarkably, the reformulation into entropic optimal transport together with these bounds allows to isolate the terms depending on the transport plan from the ones that depend on the heat kernel.  
For this reason, a recovery sequence does not need to be regularizing in terms of the heat flow, and a piecewise-constant approximation suffices to conclude $\Gamma$-convergence. No additional regularity has to be generated, which is in contrast to the dynamic setting, see Section~\ref{sec:static} for the construction. 

Up to this point, the static and dynamic formulations may be disconnected problems on metric graphs. However, they turn out to be equivalent, see Theorem~\ref{thm:equiv_static_dynamic}, thus linking the static $\Gamma$-convergence to the convergence of the minima in the dynamic case. 
The proof of this equivalence is based on the construction of solutions to the dynamic problem. To this end, we rely on the heat flow, and as this equivalence is shown for fixed $\beta>0$ without taking a small-time limit, the qualitative information and regularity of the heat kernel suffice. 

In summary, our approach compensates for the limited information provided by the weaker curvature bound. This is because the regularity requirements are shifted from the recovery sequence to the prescribed reference measure where the qualitative information is sufficient.
It is important to note that, in general, static recovery sequences lack the regularity required to be transformed into recovery sequences for the dynamic problem. This is the reason why we do not obtain $\Gamma$-convergence there.

\section{Static formulations}\label{sec:static}

In this section, we analyze the static Schrödinger problem and prove that it admits an entropic transport formulation. Moreover, we show $\Gamma$-convergence between entropic optimal transport~\eqref{eq:EntrOT_metric} and the Kantorovich formulation~\eqref{eq:static_metric_OT} in our setting. 
First, we recall known results from Polish spaces and study their implications for metric graphs. Next, we prove the convergence, closely following the arguments presented in \cite{carlier2017} that are based on a discretization of admissible plans via the so-called block-approximation. 

\subsection{Existence of solutions}

Motivated by the relations in the abstract setting, see Figure~\ref{fig:connect_diff_problems}, let $\beta>0$ and define the measure 
\begin{align*}
    R_\beta \coloneqq h_{\beta/2}\leb\otimes\leb,
\end{align*}
which satisfies $R_\beta\ll\leb\otimes\leb$. Additionally, $\leb\otimes\leb\ll R_\beta$ with density $\frac{\dx{\leb\otimes\leb}}{\dx{R_\beta}}(x,y) = h_{\beta/2}^{-1}(x,y)$ as Lemma~\ref{lem:GaussianBnds} implies $h_{\beta/2}(x,y)>0$ for a.e. $x,y\in\omgg$. 
We consider the following static Schrödinger problem for this specific reference measure
\begin{align}\label{eq:static_SBP}\tag{S$_\beta$}
    \inf\limits_{\gamma\in\Gamma(\mu_0,\mu_1)} \beta\calH(\gamma\entsep R_\beta).
\end{align}
Given two marginals with finite entropy, the above is well-posed and admits a unique solution.

\begin{proposition}\label{prop:existence_static_SBP}
    Let $\mu_0,\mu_1\in\Prob(\omgg)$ with $\calH(\mu_0\entsep\leb), \calH(\mu_1\entsep\leb)<+\infty$ and $\mu_i = \rho_i \leb$, $i=0,1$ be given. Then,~\eqref{eq:static_SBP} admits the unique solution $\gamma^\beta_{\rm opt}\in\Prob(\omgg\times\omgg)$ such that $\gamma^\beta_{\rm opt} = f^\beta\otimes g^\beta R_\beta$ for two Borel measurable functions $f^\beta,g^\beta:\omgg\to[0,\infty)$ that are unique up to transformations of the form $(f^\beta,g^\beta)\mapsto(cf^\beta,g^\beta/c)$ for a constant $c>0$. These functions satisfy the \emph{Schrödinger system }
    \begin{align}\label{eq:SchröSyst}
        \begin{cases}
            f^\beta(x)\int_{\omgg} g^\beta(y) h_{\beta/2}(x,y) \dx{y} = \rho_0(x),\\ g^\beta(x)\int_{\omgg} f^\beta(y) h_{\beta/2}(x,y) \dx{y} = \rho_1(x).
        \end{cases}
    \end{align}
    If $\rho_i \in L^\infty(\omgg)$, $i=0,1$, then $f^\beta, g^\beta \in L^\infty(\omgg)$ as well. 
\end{proposition}

\begin{proof}
    As scaling with $\beta>0$ does not change the minimizers, we consider the problem 
    \begin{align*}
        \inf\limits_{\gamma\in\Gamma(\mu_0,\mu_1)} \calH(\gamma\entsep R_\beta)
    \end{align*}
    instead. By assumption $\calH(\mu_0\otimes\mu_1\entsep R_\beta)<+\infty$ so that the infimum is finite and we can restrict to transport plans with finite entropy. 
    Since $\leb\otimes\leb\ll R_\beta\ll\leb\otimes\leb$ and $\calH(\gamma\entsep R_\beta)<+\infty$, we have that $\gamma\ll R_\beta$ as well as $\gamma\ll\leb\otimes\leb$. The existence and uniqueness of solutions now follow from an application of \cite{Tam17}[Proposition 4.1.5] with the choice $B\equiv 0$ in the assumptions, where the representation by measurable functions satisfying the Schrödinger system is concluded as well. 
    The additional $L^\infty$-bounds follow from \cite{Tam17}[Proposition 4.1.5, iv)] and the necessary assumptions are satisfied because of Lemma~\ref{lem:GaussianBnds}.
\end{proof}

\begin{remark}\label{rem:supp_fg}
    By non-negativity of $f^\beta, g^\beta$ and $h_{\beta/2}$ we have that $\rho_0(x)=0$ whenever $f^\beta(x)=0$. On the other hand, if $\rho_0(x)=0$, then either $f^\beta(x)=0$ or $\int_{\omgg} \rho^\beta(y) h_{\beta/2}(x,y) \dx{y}=0$. However, $h_{\beta/2}(x,y) > 0$ for all $x,y\in\omgg$ and therefore $g^\beta\equiv 0$ in the latter case. This contradicts the second equation from~\eqref{eq:SchröSyst} since $\mu_1\in\Prob(\omgg)$. In conclusion $\supp(f^\beta) = \supp(\rho_0)$ and by analogous arguments $\supp(g^\beta) = \supp(\rho_1)$.
\end{remark}

Similar to the known equivalences, the static Schrödinger problem~\eqref{eq:static_SBP} admits a reformulation as entropic optimal transport on metric graphs. In this setting, however, the cost function differs from the quadratic distance due to the structure of the heat kernel. The resulting formulation reads as follows. 

\begin{proposition}
    For given $\mu_0,\mu_1\in\Prob(\omgg)$ and $\beta>0$, the problem~\eqref{eq:static_SBP} is equivalent to
    \begin{align}\label{eq:EntrOT}\tag{EOT}
    \inf\limits_{\gamma\in\Gamma(\mu_0,\mu_1)} \iint_{\omgg\times\omgg} d_\sfG^2(x,y) - \beta \log\bra*{\frac{h_{\beta/2}(x,y)}{k_{\beta/2}(d_\sfG(x,y))}} \dx{\gamma(x,y)} + \beta \calH(\gamma\entsep\leb\otimes\leb) + \frac{\beta\log(\beta \pi)}{2}
\end{align}
and this entropic transport problem admits a unique minimizer.
\end{proposition}

\begin{proof}
    If $\gamma\in\Gamma(\mu_0,\mu_1)$ is such that $\gamma\centernot\ll R_\beta$, then $\gamma\centernot\ll \leb\otimes\leb$ since $\leb\otimes\leb\ll R_\beta$ and we have that $\calH(\gamma\entsep R_\beta) = +\infty = \calH(\gamma\entsep\leb\otimes\leb)$.
    Therefore, we only need to consider $\gamma\ll R_\beta$. In this case $\gamma\ll\leb\otimes\leb$ as well and $h_{\beta/2}>0$ by Lemma~\ref{lem:GaussianBnds}. We can rewrite the entropy as
    \begin{align*}
        &\calH(\gamma\entsep R_\beta) = \iint_{\omgg\times\omgg} \log\bra*{\frac{\dx{\gamma}}{\dx{R_\beta}}} \dx{\gamma} = \iint_{\omgg\times\omgg} \log\bra*{\frac{\dx{\gamma}}{\dx{\leb\otimes\leb}}(x,y) h_{\beta/2}^{-1}(x,y)} \dx{\gamma(x,y)}\\
        &= \calH(\gamma\entsep\leb\otimes\leb) - \iint_{\omgg\times\omgg} \log\bra*{h_{\beta/2}(x,y)} \dx{\gamma(x,y)}\\
        &= \calH(\gamma\entsep\leb\otimes\leb) - \iint_{\omgg\times\omgg} \log\bra*{k_{\beta/2}(d_\sfG(x,y))} \dx{\gamma(x,y)} - \iint_{\omgg\times\omgg} \log\bra*{\frac{h_{\beta/2}(x,y)}{k_{\beta/2}(d_\sfG(x,y))} } \dx{\gamma(x,y)}.
    \end{align*}
    Moreover, it holds that
    \begin{align*}
        \log\bra*{k_{\beta/2}(d_\sfG(x,y))} &= -\frac{1}{2}\log(2\pi\beta) - \frac{d^2_\sfG(x,y)}{2\beta}.
    \end{align*}
    Multiplying with $\beta>0$ then gives
    \begin{align*}
        \beta\calH(\gamma\entsep R_\beta) &= \beta\calH(\gamma\entsep\leb\otimes\leb) + \iint_{\omgg\times\omgg} \frac{1}{2}d^2_\sfG(x,y) - \beta\log\bra*{\frac{h_{\beta/2}(x,y)}{k_{\beta/2}(d_\sfG(x,y))} } \dx{\gamma(x,y)} + \frac{\beta\log(2\beta \pi)}{2}.
    \end{align*}
    The equivalence now follows after applying the infimum to both sides. 
    
    For the well-posedness of~\eqref{eq:EntrOT}, note that the functional is weakly lower semi-continuous in $\gamma$, thus, the direct method of Calculus of Variations is applicable as in \cite{santambrogio_optimal_2015}[Theorem 1.7]. By strict convexity of $\gamma\mapsto\calH(\gamma\entsep\leb\otimes\leb)$ and linearity of all remaining terms, the minimizer is unique. 
\end{proof}

\subsection{\texorpdfstring{$\Gamma$}{Gamma}-convergence of the static problems}

The entropic optimal transport problem~\eqref{eq:EntrOT} and the static Kantorovich formulation~\eqref{eq:static_OT} are connected by means of $\Gamma$-convergence of functionals, see Definition~\ref{def:Gammcvg}.
In order to show this relation, two asymptotic bounds have to be verified and we consider the asymptotic lower bound first.
Note that the relative entropy is lower semicontinuous with respect to weak convergence of measures and non-negative by Jensen's inequality \cite[Section 5]{MatthiasErbar2022}. From this observation we infer the following. 

\begin{lemma}[Liminf-estimate]\label{lem:LimInfEst}
    For $\gamma, \gamma_\beta\in\Prob(\omgg\times\omgg)$, $\beta>0$, such that $\gamma_\beta\weakcvg\gamma$ as $\beta\to 0$ it holds that
    \begin{align*}
        \iint_{\omgg\times\omgg} \frac{1}{2}d^2_\sfG(x,y) \dx{\gamma(x,y)}  \leq \liminf\limits_{\beta\to 0}\Bigg[&\iint_{\omgg\times\omgg} \frac{1}{2}d^2_\sfG(x,y) \dx{\gamma_\beta(x,y)} + \beta \calH(\gamma_\beta \entsep \leb\otimes\leb) \\&- \beta \iint_{\omgg\times\omgg} \log\bra*{\frac{h_{\beta/2}(x,y)}{k_{\beta/2}(d_\sfG(x,y))}} \dx{\gamma_\beta(x,y)} + \frac{\beta\log(2\beta \pi)}{2}\Bigg].
    \end{align*}
\end{lemma}

\begin{proof}
    Note that $\lim_{r\searrow 0} r \log(r) = 0$ and that, by lower semi-continuity of the cost-functional and non-negativity of the entropy, we have 
    \begin{align}
        \iint_{\omgg\times\omgg} \frac{1}{2}d^2_\sfG(x,y) \dx{\gamma(x,y)} &\leq \liminf\limits_{\beta\to 0} \iint_{\omgg\times\omgg} \frac{1}{2}d^2_\sfG(x,y) \dx{\gamma_\beta(x,y)}\notag\\
        &\leq \liminf\limits_{\beta\to 0}\bra*{\iint_{\omgg\times\omgg} \frac{1}{2}d^2_\sfG(x,y) \dx{\gamma_\beta(x,y)} + \beta \calH(\gamma_\beta \entsep \leb\otimes\leb) + \frac{\beta\log(2\beta \pi)}{2}}.\label{eq:LiminfStep1}
    \end{align}
    For the last term without loss of generality assume $\beta\leq 2$. Lemma~\ref{lem:GaussianBnds} with $T=1$ implies that 
    \begin{align*}
        \log\bra*{\frac{h_{\beta/2}(x,y)}{k_{\beta/2}(d_\sfG(x,y))}} \in [\log(C_0),\log(C_1)] \quad \text{for a.e. } x,y\in\sfG
    \end{align*}
    independent of $\beta>0$ small enough. Since $\gamma_\beta\in \Prob(\omgg\times\omgg)$ for all $\beta>0$, this gives
    \begin{align}\label{eq:LimIs0}
        0 = \lim\limits_{\beta\to 0} \beta \log(C_1) \geq \lim\limits_{\beta\to 0 } \beta \iint_{\omgg\times\omgg} \log\bra*{\frac{h_{\beta/2}(x,y)}{k_{\beta/2}(d_\sfG(x,y))}} \dx{\gamma_\beta(x,y)} &\geq \lim\limits_{\beta\to 0} \beta \log(C_0) = 0
    \end{align}
    and therefore the limit vanishes. Adding~\eqref{eq:LimIs0} to the right-hand side of~\eqref{eq:LiminfStep1} results in
    \begin{align*}
        \iint_{\omgg\times\omgg} \frac{1}{2}d^2_\sfG(x,y) \dx{\gamma(x,y)} \leq \liminf\limits_{\beta\to 0}\Bigg(\iint_{\omgg\times\omgg} \frac{1}{2}d^2_\sfG(x,y) \dx{\gamma_\beta(x,y)} + \beta \calH(\gamma_\beta \entsep \leb\otimes\leb) + \frac{\beta\log(2\beta \pi)}{2}\Bigg)&\\
        - \lim\limits_{\beta\to 0 } \beta \iint_{\omgg\times\omgg} \log\bra*{\frac{h_{\beta/2}(x,y)}{k_{\beta/2}(d_\sfG(x,y))}} \dx{\gamma_\beta(x,y)}&\\
        \leq \liminf\limits_{\beta\to 0}\Bigg(\iint_{\omgg\times\omgg} \frac{1}{2}d^2_\sfG(x,y) \dx{\gamma_\beta(x,y)} + \beta \calH(\gamma_\beta \entsep \leb\otimes\leb) + \frac{\beta\log(2\beta \pi)}{2}&\\ - \beta \iint_{\omgg\times\omgg} \log\bra*{\frac{h_{\beta/2}(x,y)}{k_{\beta/2}(d_\sfG(x,y))}} \dx{\gamma_\beta(x,y)}\Bigg)&.
    \end{align*}
\end{proof}

It remains to show the asymptotic upper bound. To this end we rely on the so-called block approximation introduced in \cite{carlier2017} and adapt the arguments to our modified cost functions. 

\begin{definition}[Block Approximation]
    Let $\mu_0,\mu_1\in\Prob(\omgg)$ with $\calH(\mu_0\entsep\leb),\calH(\mu_1\entsep\leb) < +\infty$ and $\gamma\in\Gamma(\mu_0,\mu_1)$ be given. For $h>0$ and $K\in\N$ let $\calT_h = \set{T^1_h,\ldots,T^K_h}$ be a \emph{partition} of $\omgg$, that is for all $k\in\set{1,\ldots,K}$ and $l\neq k$ it holds that
    \begin{align*}
        T^k_h \subset \omgg, \quad 0<\diam(T_h^k)<h, \quad T^k_h\cap T^l_h = \emptyset \quad \text{and} \quad \bigcup\limits_{k=1}^K T^k_h = \omgg.
    \end{align*}
    For a given partition $\calT_h$ we set $K(\calT_h) \coloneqq \abs{\calT_h}$.
    The \emph{block approximation} of $\gamma$ with respect to $\calT_h$ is 
    \begin{align*}
        \gamma_h \coloneqq \sum\limits_{k,l=1}^{K(\calT_h)} \gamma(T^k_h\times T^l_h)(\mu_0^k\otimes \mu_1^l),
    \end{align*}
    where for every Borel measurable set $A\subset\omgg$ we define 
    \begin{align*}
        \mu_0^k(A) \coloneqq \begin{cases}
            \frac{\mu_0(A\cap T^k_h)}{\mu_0(T^k_h)} &: \mu_0(T^k_h)>0\\ 0 &:\text{else}
        \end{cases} \quad\text{and}\quad \mu_1^l(A) \coloneqq \begin{cases}
            \frac{\mu_1(A\cap T^l_h)}{\mu_1(T^l_h)} &: \mu_1(T^l_h)>0\\ 0 &:\text{else}
        \end{cases}.
    \end{align*}
\end{definition}

The following useful properties can be verified as in \cite[Corollary 2.12]{carlier2017}. 

\begin{proposition}\label{prop:PropBlockAp}
    Let $\mu_0,\mu_1\in\Prob(\omgg)$ be given such that $\calH(\mu_0\entsep\leb),\calH(\mu_1\entsep\leb) < +\infty$ with $\mu_i = \rho_i\leb$ for $i=0,1$ and consider $\gamma\in\Gamma(\mu_0,\mu_1)$. For a partition $\calT_h$ of $\omgg$ let $\gamma_h$ be the corresponding block approximation. Then, $\gamma_h\in \Gamma(\mu_0,\mu_1)$ and $\gamma_h\ll\leb\otimes\leb$ with density 
    \begin{align*}
        \gamma_h(x,y) = \begin{cases}
            \gamma(T^k_h\times T^l_h) \frac{\rho_0(x)}{\mu_0(T^k_h)}\frac{\rho_1(y)}{\mu_1(T^l_h)} &: \mu_0(T^k_h), \mu_1(T^l_h)> 0\\ 0 &:\text{else}
        \end{cases}
    \end{align*}
    for $x\in T^k_h$ and $y\in T^l_h$. Moreover, given a sequence of partitions $(\calT_h)_{h>0}$ it holds that $\gamma_h\weakcvg\gamma$ as $h\to 0$. 
\end{proposition}

We are now in a position to prove the $\limsup$-estimate. For the entropy, we use the fact that metric graphs are compact and $\leb(\omgg)<\infty$, thus resulting in a slightly simplified argument compared to \cite[Proposition 2.14]{carlier2017}. 

\begin{proposition}[Limsup-estimate]\label{prop:LimSupEst}
    Let $\mu_0,\mu_1\in\Prob(\omgg)$ with $\calH(\mu_0\entsep\leb), \calH(\mu_1\entsep\leb)<+\infty$. Then, for any $\gamma\in\Gamma(\mu_0,\mu_1)$ there exists a sequence $(\gamma_\beta)_{\beta>0}\subset\Gamma(\mu_0,\mu_1)$ with $\gamma^\beta\weakcvg\gamma$ as $\beta\to0$ and
    \begin{align*}
        \iint_{\omgg\times\omgg} \frac{1}{2}d^2_\sfG(x,y) \dx{\gamma(x,y)}  \geq \limsup\limits_{\beta\to0} \Bigg(\iint_{\omgg\times\omgg} \frac{1}{2}d_\sfG^2(x,y&) \dx{\gamma^\beta(x,y)} + \beta \calH(\gamma^\beta\entsep\leb\otimes\leb) + \frac{\beta\log(2\beta \pi)}{2} \\&- \beta \iint_{\omgg\times\omgg} \log\bra*{\frac{h_{\beta/2}(x,y)}{k_{\beta/2}(d_\sfG(x,y))}} \dx{\gamma_\beta(x,y)}\Bigg).
    \end{align*}
\end{proposition}

\begin{proof}
    Let $\ell_{\min} \coloneqq \min\set{\ell_\sfe \setsep\sfe\in\sfE}$ be the length of the shortest edge and without loss of generality, assume that $\beta\in(0,\ell_{\min}/2)$. For each $\sfv\in\sfV$ let $\overline{B_\beta(\sfv)} \subset\omgg$ be the closed ball of radius $\beta>0$ around $\sfv$, and let $(I^{\sfe,\beta}_k)_{k=1,\ldots,\lceil(\ell_\sfe-2\beta)/\beta\rceil}$ be a partition of $[\beta,\ell_\sfe-\beta]$ for $\sfe\in\sfE$, where $I^{\sfe,\beta}_k \in \set{(a^{\sfe,\beta}_k,b^{\sfe,\beta}_k), (a^{\sfe,\beta}_k,b^{\sfe,\beta}_k], [a^{\sfe,\beta}_k,b^{\sfe,\beta}_k), [a^{\sfe,\beta}_k,b^{\sfe,\beta}_k]}$ for $a^{\sfe,\beta}_k<b^{\sfe,\beta}_k$. Then 
    $$
    \calT_\beta \coloneqq\set*{\overline{B_\beta(\sfv)} \setsep \sfv\in\sfV}\cup\set*{I^{\sfe,\beta}_k \setsep \sfe\in\sfE, \, k\in\set{1,\ldots, \lceil(\ell_\sfe-2\beta)/\beta\rceil}}
    $$
    defines an admissible partition of scale $\beta$. Denote by $\gamma_\beta\in\Prob(\omgg\times\omgg)$ the corresponding block approximation.
    By Proposition~\ref{prop:PropBlockAp} we have $\gamma_\beta\weakcvg\gamma$ as $\beta\to0$, and since the distance function is continuous on $\sfG$ this implies 
    \begin{align}\label{eq:bnd_action}
        \iint_{\omgg\times\omgg} d^2_\sfG(x,y)\dx{\gamma_\beta(x,y)} \overset{\beta\to0}{\longrightarrow} \iint_{\omgg\times\omgg} d^2_\sfG(x,y)\dx{\gamma(x,y)}.
    \end{align}
    By assumption $\calH(\mu_i\entsep\leb)<+\infty$, therefore $\mu_i = \rho_i\leb$ for $i=0,1$.
    Further, note that $\gamma(T^k_\beta\times T^l_\beta) \leq 1$, and therefore $\log\bra*{\gamma(T^k_\beta\times T^l_\beta)}\leq 0$, which gives
    \begin{align*}
        &\calH(\gamma_\beta\entsep\leb\otimes\leb) = \sum\limits_{k,l=1}^{N(\calT_\beta)}\iint_{T^k_\beta\times T^l_\beta} \gamma(T^k_\beta\times T^l_\beta) \frac{\rho_0(x)}{\mu_0(T^k_\beta)}\frac{\rho_1(y)}{\mu_1(T^l_\beta)}  \log\bra*{\gamma(T^k_\beta\times T^l_\beta) \frac{\rho_0(x)}{\mu_0(T^k_\beta)}\frac{\rho_1(y)}{\mu_1(T^l_\beta)} } \dx{x}\dx{y}\\
        &=\sum\limits_{k,l=1}^{N(\calT_\beta)}\int_{T^k_\beta} \int_{T^l_\beta} \gamma(T^k_\beta\times T^l_\beta) \frac{\rho_0(x)}{\mu_0(T^k_\beta)}\frac{\rho_1(y)}{\mu_1(T^l_\beta)}\bra*{ \log\bra*{\gamma(T^k_\beta\times T^l_\beta)} + \log\bra*{\frac{\rho_0(x)}{\mu_0(T^k_\beta)}} +\log\bra*{\frac{\rho_1(y)}{\mu_1(T^l_\beta)} } }\dx{x}\dx{y}\\
        &\leq \sum\limits_{k=1}^{N(\calT_\beta)} \int_{T^k_\beta} \rho_0(x)\log\bra*{\frac{\rho_0(x)}{\mu_0(T^k_\beta)}}\dx{x} + \sum\limits_{l=1}^{N(\calT_\beta)}   \int_{T^l_\beta} \rho_1(y)\log\bra*{\frac{\rho_1(x)}{\mu_1(T^l_\beta)}}\dx{y}\\
        &= \calH(\mu_0\entsep\leb) - \sum\limits_{k=1}^{N(\calT_\beta)} \mu_0(T^k_\beta)\log\bra*{\mu_0(T^k_\beta)} + \calH(\mu_1\entsep\leb) - \sum\limits_{l=1}^{N(\calT_\beta)} \mu_1(T^l_\beta)\log\bra*{\mu_1(T^l_\beta)}.
    \end{align*}
    Since $r\mapsto r\log r$ is bounded from below by $-e^{-1}$, and since $N(\calT^h) \leq \abs{\sfV} + \leb(\omgg)/\beta$ by construction we obtain the upper bound
    \begin{align}\label{eq:bnd_entr}
        \calH(\gamma_\beta\entsep\leb\otimes\leb) &\leq \calH(\mu_0\entsep\leb) + \calH(\mu_1\entsep\leb) + 2e^{-1}\bra*{\abs{\sfV} + \frac{\leb(\omgg)}{\beta}}.
    \end{align}
    We define $\gamma^\beta\coloneqq\gamma_{\sqrt{\beta}}$, and combining~\eqref{eq:bnd_action} with~\eqref{eq:bnd_entr} we get 
    \begin{align*}
        &\limsup\limits_{\beta\to0} \bra*{\iint_{\omgg\times\omgg} \frac{1}{2}d_\sfG^2(x,y) \dx{\gamma_\beta(x,y)} + \beta \calH(\gamma_\beta\entsep\leb\otimes\leb) + \frac{\beta\log(2\beta \pi)}{2}}\\
        &\leq \iint_{\omgg\times\omgg} \frac{1}{2}d^2_\sfG(x,y) \dx{\gamma(x,y)} + \limsup\limits_{\beta\to0}\beta\bra*{\calH(\mu_0\entsep\leb) + \calH(\mu_1\entsep\leb) + 2e^{-1}\bra*{\abs{\sfV} + \frac{\leb(\omgg)}{\sqrt{\beta}}} }\\
        &= \iint_{\omgg\times\omgg} \frac{1}{2}d^2_\sfG(x,y) \dx{\gamma(x,y)}.
    \end{align*}
    Since $\gamma^\beta\weakcvg\gamma$ as $\beta\to 0$ with $\gamma^\beta\in\Gamma(\mu_0,\mu_1)$ for all $\beta\in (0,\ell_{\min}/2)$,~\eqref{eq:LimIs0} holds true for this particular sequence.
    From $\lim_{r\to 0} r\log(r) = 0$ we conclude that
    \begin{align*}
        &\limsup\limits_{\beta\to0} \bra*{\iint_{\omgg\times\omgg} \frac{1}{2}d_\sfG^2(x,y) \dx{\gamma_\beta(x,y)} + \beta \calH(\gamma_\beta\entsep\leb\otimes\leb) + \frac{\beta\log(2\beta \pi)}{2} - \beta \iint_{\omgg\times\omgg} \log\bra*{\frac{h_{\beta/2}(x,y)}{k_{\beta/2}(d_\sfG(x,y))}} \dx{\gamma_\beta(x,y)}} \\
        &\leq \limsup\limits_{\beta\to0} \bra*{\iint_{\omgg\times\omgg} \frac{1}{2}d_\sfG^2(x,y) \dx{\gamma_\beta(x,y)} + \beta \calH(\gamma_\beta\entsep\leb\otimes\leb) + \frac{\beta\log(2\beta \pi)}{2}} \\&- \liminf\limits_{\beta\to0} \beta \iint_{\omgg\times\omgg} \log\bra*{\frac{h_{\beta/2}(x,y)}{k_{\beta/2}(d_\sfG(x,y))}} \dx{\gamma_\beta(x,y)}
        = \iint_{\omgg\times\omgg} \frac{1}{2}d^2_\sfG(x,y) \dx{\gamma(x,y)}.
    \end{align*}
\end{proof}

Having established both estimates, the $\Gamma$-convergence follows. 

\begin{theorem}\label{thm:gamma_cvg_static}
    For $\mu_0,\mu_1\in\Prob(\omgg)$ with $\calH(\mu_0\entsep\leb),\calH(\mu_1\entsep\leb)<+\infty$, it holds that~\eqref{eq:EntrOT} $\Gamma$-converges to~\eqref{eq:static_OT} as $\beta\to 0$. Additionally, minimizers of~\eqref{eq:EntrOT} converge weakly to minimizers of~\eqref{eq:static_OT} as well as the minimum values.
\end{theorem}

\begin{proof}
    The $\Gamma$-convergence is a direct consequence of Lemma~\ref{lem:LimInfEst} together with Proposition~\ref{prop:LimSupEst}. Recalling that the minimizer of~\eqref{eq:static_SBP} is unique by Proposition~\ref{prop:existence_static_SBP}, the convergence of minimizers and the minimum values follows from \cite{dal_maso_introduction_1993}[Proposition 7.18], where equi-coercivity holds due to the compactness of $\Gamma(\mu_0,\mu_1)$ in the weak topology. 
\end{proof}

\section{The dynamic Schrödinger problem}\label{sec:dynamic}

Having established a rigorous connection between~\eqref{eq:EntrOT} and~\eqref{eq:static_OT}, we now look at the dynamic formulations. First, we derive the dynamic Schrödinger problem from the static one under additional regularity assumptions on initial and final data and show the equivalence to the static one as well as uniqueness of solutions. Then, we extend the found formulation to a broader class and study well-posedness in this case. 
The main result of this section is contained in Theorem~\ref{thm:equiv_static_dynamic}, where we prove the equivalence between the static Schrödinger problem~\eqref{eq:static_SBP} and a dynamic formulation. This is in analogy to the optimal transport case on metric graphs, where a Benamou-Brenier formulation has been derived \cite[Corollary 1]{MatthiasErbar2022}.

In order to formulate the dynamic version of~\eqref{eq:static_SBP}, we first specify the constraint set and introduce our notion of solutions to the continuity equation. 
\begin{definition}\label{def:CE_CEb}
    Let $\mu_0,\mu_1\in\Prob(\omgg)$ be given. A pair of measures $(\mu_t, J_t) \in \Prob(\omgg)\times \Meas(\omge)$ for $t\in(0,1)$ is called \emph{weak solution to the continuity equation} if the following hold true.
    \begin{enumerate}[label=\roman*)]
        \item The map $t\mapsto \mu_t$ is weakly continuous,
        \item $\int_0^1 \abs{J_t}\dx{t} < +\infty$, and
        \item for all test functions $\varphi\in C^1(\omgg)$ and a.e. $t\in(0,1)$, we have
        \begin{align*}
            \frac{\dx{}}{\dx{t}} \int_{\omgg} \varphi \dx{\mu_t} &= \int_{\omge} \nabla\varphi\cdot \dx{J_t}.
        \end{align*}
    \end{enumerate}
    We write $(\mu_t,J_t)\in\ceunreg$ for the set of weak solutions to the continuity equation that satisfy 
    \begin{align*}
        \mu_t\weakcvg\mu_i \quad \text{as } t\to i, \text{ for } i=0,1. 
    \end{align*}
    We call $(\mu_t, J_t)$ a \emph{bounded solution to the continuity equation}, if additionally
    \begin{enumerate}[label=\roman*)]
        \setcounter{enumi}{3}
        \item there exists a constant $C>0$ such that $\mu_t \leq C \leb$ for all $t\in[0,1]$.
    \end{enumerate}
    We denote by 
    \begin{align*}
        \cebound \coloneqq \set*{(\mu_t, J_t)\in \ceunreg \setsep (\mu_t, J_t) \text{ is a bounded solution}}.
    \end{align*}
\end{definition}

With this notation at hand, the dynamic optimal transport problem on metric graphs can be written as 
\begin{align}\label{eq:OT_dynamic}\tag{$\rm D_{OT}$}
    \frac{1}{2}W^2_2(\mu_0,\mu_1) = \inf\limits_{(\mu_t, v_t \mu_t) \in \ceunreg} \int_0^1 \int_{\omge} \frac{1}{2}\abs{v_t}^2 \dx{\mu_t} \dx{t}.
\end{align}

\subsection{Derivation from the static problem}

To derive the dynamic Schrödinger problem and show its equivalence to the static one, we mimic the rigorous calculations made in \cite[Chapter 5]{Tam17} in ${\rm RCD}^*(K,N)$-spaces. However, as metric graphs do not fall into this category, the regularity properties established in the abstract setting are not applicable. Instead, we rely on the explicit structure of the heat kernel given in Lemma~\ref{lem:RegHeat}. Additionally, we construct solutions to the continuity equation based on the following observation. 

Recall that, by Proposition~\ref{prop:existence_static_SBP} the unique solution of~\eqref{eq:static_SBP} is of the form $h_{\beta/2} (f^\beta\leb)\otimes (g^\beta\leb)$ for $\beta>0$ and two Borel measurable functions $f^\beta,g^\beta:\omgg \to [0,\infty)$.
Formally, the static Schrödinger problem~\eqref{eq:static_SBP} admits a dual version, namely 
\begin{align}\label{eq:dualSBP}\tag{D}
    \beta\sup\limits_{\varphi,\psi\in C(\omgg)} \int_{\omgg} \varphi \dx{\mu_0} + \int_{\omgg}\psi\dx{\mu_1} - \log\iint_{\omgg\times\omgg} e^{\varphi\oplus \psi} \dx{R_\beta}.
\end{align}
This dual problem admits a unique solution as well, which relates to the primal one via $\varphi = \log f^\beta$ and $\psi=\log g^\beta$. In general, these functions may not be continuous on $\sfG$.
Moreover, the functions $f^\beta$ and $g^\beta$ are merely measurable and might not be integrable. To overcome this issue, we make use of the propagation of the $L^\infty$-bounds from Proposition~\ref{prop:existence_static_SBP}, available under the following assumption. 

\begin{assumption}\label{as:Linfty}
    Throughout this section, assume that $\mu_i\in\Prob(\omgg)$ is given such that $\mu_i=\rho_i\leb$ with density $\rho_i \in L^\infty(\omgg)$ for $i=0,1$. 
\end{assumption}
By Proposition~\ref{prop:existence_static_SBP}, Assumption~\ref{as:Linfty} implies $f^\beta, g^\beta \in L^\infty(\omgg)$ for $\gamma^\beta_{\rm opt} = f^\beta \otimes g^\beta R_\beta$ the solution of~\eqref{eq:static_SBP}. As $\omgg$ is compact, this implies $f^\beta, g^\beta \in L^p(\omgg)$ for all $p\in[1,\infty]$ as well.

In order to derive the dynamic Schrödinger problem, we introduce the quantities
\begin{align}\label{eq:ce_from_fg}
    \begin{cases}
        f_t^\beta \coloneqq H_{\beta t/2} f^\beta\\
        \varphi_t^\beta \coloneqq \beta \log f_t^\beta
    \end{cases}, \quad \begin{cases}
        g_t^\beta\coloneqq H_{\beta(1-t)/2}g^\beta\\
        \psi_t^\beta\coloneqq\beta \log g_t^\beta
    \end{cases}, \quad \begin{cases}
        \rho_t^\beta \coloneqq f_t^\beta g_t^\beta\\
        \mu_t^\beta \coloneqq \rho_t^\beta \leb\\
        \vartheta_t^\beta\coloneqq(\psi_t^\beta - \varphi_t^\beta)/2
    \end{cases} 
\end{align}
for $t\in(0,1)$ and set
\begin{align*}
    f_0^\beta = f^\beta, \quad g_1^\beta= g^\beta, \quad \rho_0^\beta =\rho_0, \quad \rho_1^\beta=\rho_1, \quad \varphi^\beta_0 = \chi_{\supp(\rho_0)} \beta \log f^\beta , \quad \psi_0^\beta = \chi_{\supp(\rho_1)} \beta\log g^\beta .
\end{align*}
Note that for any $t>0$ we have $f_t^\beta, g_t^\beta >0$ by Lemma~\ref{lem:RegHeat}, which is why $\varphi^\beta_t, \psi^\beta_t$ and $\vartheta_t^\beta$ are well-defined for positive times. Because of Remark~\ref{rem:supp_fg} and using the extension $0\cdot\log0=0$, the initial and final conditions are well-defined as well.
Additionally, the functions from~\eqref{eq:ce_from_fg} inherit further regularity from the heat kernel. 

\begin{lemma}\label{lem:Regularity_fg}
    Let $\mu_0,\mu_1\in \Prob(\omgg)$ satisfy Assumption~\ref{as:Linfty}. Then
    $$
    f_t^\beta, g_t^\beta, \rho_t^\beta, \varphi_t^\beta, \psi_t^\beta, \vartheta_t^\beta\in C^1((0,1))\cap C^\infty(\omge) 
    $$ 
    with the notation introduced in~\eqref{eq:Notation_C1Cinfty}. The quantities defined in~\eqref{eq:ce_from_fg} satisfy the following equations pointwise on each edge 
    \begin{equation}\label{eq:formal_eq_fg}
    \begin{aligned}
        &\frac{\dx{}}{\dx{t}} f_t^\beta = \frac{\beta}{2} \Delta f_t^\beta, &&\frac{\dx{}}{\dx{t}} g_t^\beta = -\frac{\beta}{2} \Delta g_t^\beta,\\
        &\frac{\dx{}}{\dx{t}}  \varphi_t^\beta = \frac{1}{2}\abs{\nabla\varphi_t^\beta}^2 + \frac{\beta}{2}\Delta \varphi_t^\beta, && -\frac{\dx{}}{\dx{t}} \psi_t^\beta = \frac{1}{2}\abs{\nabla\psi_t^\beta}^2 + \frac{\beta}{2}\Delta \psi_t^\beta,\\
        &\frac{\dx{}}{\dx{t}} \rho_t^\beta + \nabla\cdot\bra{\rho_t^\beta \nabla\vartheta_t^\beta} = 0, &&\frac{\dx{}}{\dx{t}} \vartheta_t^\beta + \frac{1}{2}\abs{\nabla\vartheta_t^\beta}^2 = - \frac{\beta^2}{8}\bra*{2\Delta\log\rho_t^\beta + \abs{\nabla \log \rho_t^\beta}^2}
        \end{aligned}
    \end{equation}
    for all  $t\in(0,1)$. Moreover, $f^\beta_t, g^\beta_t, \rho_t^\beta, \varphi_t^\beta, \psi_t^\beta \in C(\omgg)$ and the following coupling conditions hold
    \begin{equation}\label{eq:BCs}
    \begin{aligned}
        &\sum\limits_{\sfe\in\sfE(\sfv)} \nabla f^{\beta,\sfe}_t(\sfv) \cdot n^{\sfe}_\sfv = 0, &&\sum\limits_{\sfe\in\sfE(\sfv)} \nabla g^{\beta,\sfe}_t(\sfv) \cdot n^{\sfe}_\sfv = 0,
        &&\sum\limits_{\sfe\in\sfE(\sfv)} \nabla \varphi^{\beta,\sfe}_t(\sfv) \cdot n^{\sfe}_\sfv = 0,\\
        & \sum\limits_{\sfe\in\sfE(\sfv)} \nabla \psi^{\beta,\sfe}_t(\sfv) \cdot n^{\sfe}_\sfv = 0,
        &&\sum\limits_{\sfe\in\sfE(\sfv)} \nabla \rho^{\beta,\sfe}_t(\sfv) \cdot n^{\sfe}_\sfv = 0, &&\sum\limits_{\sfe\in\sfE(\sfv)} \nabla \vartheta^{\beta,\sfe}_t(\sfv) \cdot n^{\sfe}_\sfv = 0.
        \end{aligned}
    \end{equation}
\end{lemma}

\begin{proof}
    Since $\omgg$ is compact, we have $L^\infty(\omgg)\hookrightarrow L^p(\omgg)$ for $p\in[1,\infty)$. Then, the regularity of $f_t^\beta$ and $g_t^\beta$ is direct a consequence of their definition and Lemma~\ref{lem:RegHeat}. The same holds for $\varphi_t^\beta, \psi_t^\beta, \vartheta_t^\beta$.

    The equations from~\eqref{eq:formal_eq_fg} can be verified by direct computations, which are possible due to the already established regularity. 
    The continuity and coupling conditions~\eqref{eq:BCs} on the other hand are a consequence of the definition of the heat kernel. 
\end{proof}
These properties suffice to rigorously adapt the calculations from \cite[Section 5.4]{Tam17} on metric graphs and to prove the main result of this section. In order to allow for initial and final data with bounded support, we additionally use the arguments presented in \cite[Theorem 4.1]{gigli_benamou-brenier_2018}.

\begin{theorem}\label{thm:equiv_static_dynamic}
    For $\mu_0, \mu_1 \in \Prob(\omgg)$ with $\calH(\mu_0\entsep\leb), \calH(\mu_1\entsep \leb) < +\infty$ satisfying Assumption~\ref{as:Linfty}, the static Schrödinger problem~\eqref{eq:static_SBP} is equivalent to the following dynamic formulation 
    \begin{align}\tag{$\rm D_{SP}$}\label{eq:BB_SBP}
        \frac{\beta}{2}\bra*{\calH(\mu_0\entsep\leb) + \calH(\mu_1\entsep\leb)} + \inf\limits_{(\rho_t\leb,v_t\rho_t\leb)\in\cebound} \int_0^1 \int_{\omge} \bra*{\frac{1}{2} \abs{v_t}^2 + \frac{\beta^2}{8}\abs{\nabla\log \rho_t}^2 } \rho_t\dx{x}\dx{t}
    \end{align}
    whose unique minimizer is $(\rho_t^\beta, \nabla\vartheta_t^\beta \rho_t^\beta) \in \cebound$. 
\end{theorem}

\begin{proof}
    The proof consists of three parts. In the first step, we show that there exists a pair $(\rho_t\leb, v_t\rho_t\leb)\in \ceunreg$ of solutions to the continuity equation such that 
    \begin{align}\label{eq:equal_competitor}
        \inf\limits_{\gamma\in\Gamma(\mu_0,\mu_1)} \beta\calH(\gamma\entsep R_\beta) &= \frac{\beta}{2}\bra*{\calH(\mu_0\entsep\leb) + \calH(\mu_1\entsep\leb)} + \int_0^1 \int_{\omge} \bra*{\frac{1}{2} \abs{v_t}^2 + \frac{\beta^2}{8}\abs{\nabla\log \rho_t}^2 } \rho_t\dx{x}\dx{t}.
    \end{align}
    Next, we prove that this choice is optimal in~\eqref{eq:BB_SBP}, thus attaining the infimum, and subsequently, we show the uniqueness of the solution.

    \textbf{Step 1:}
    From~\eqref{eq:formal_eq_fg}, we already know that $(\rho_t\leb, \nabla\vartheta_t\rho_t\leb)\in \cebound$ for $\rho_t^\beta$ and $\vartheta_t^\beta$ defined in~\eqref{eq:ce_from_fg}. Next, we show that this choice attains the equality in~\eqref{eq:equal_competitor}.
    
    Since $\vartheta_t^\beta \rho_t^\beta \in C^1((0,1))\cap C^\infty(\omge)$ by Lemma~\ref{lem:Regularity_fg} solving the respective equations~\eqref{eq:formal_eq_fg} pointwise in $\supp(\rho_0)\cap\supp(\rho_1)$, their product is differentiable in time with derivative 
    \begin{align*}
        \frac{\dx{}}{\dx{t}} (\vartheta_t^\beta \rho_t^\beta) &= \bra*{- \frac{1}{2}\abs{\nabla\vartheta_t^\beta}^2 - \frac{\beta^2}{8}\bra*{2\Delta\log\rho_t^\beta + \abs{\nabla \log \rho_t^\beta}^2}} \rho_t^\beta - \vartheta_t^\beta \nabla\cdot\bra{\rho_t^\beta \nabla \vartheta_t^\beta} \in C(\omge) 
    \end{align*}
    for all $t\in(0,1)$. Additionally, $C(\omge) \hookrightarrow L^1(\omge)$ since $\omge$ is bounded.
    Therefore, we can apply dominated convergence, which yields
    \begin{align*}
        \frac{\dx{}}{\dx{t}}\int_{\omgg} \vartheta_t^\beta \rho_t^\beta \dx{x} &= \int_{\omge} \bra*{- \frac{1}{2}\abs{\nabla\vartheta_t^\beta}^2 - \frac{\beta^2}{8}\bra*{2\Delta\log\rho_t^\beta + \abs{\nabla \log \rho_t^\beta}^2}}\rho_t^\beta \dx{x} - \int_{\omge} \vartheta_t^\beta\nabla\cdot\bra{\rho_t^\beta \nabla \vartheta_t^\beta} \dx{x}
    \end{align*}
    for $t\in(0,1)$.
    After integrating by parts, observing that $\nabla \log\rho_t^\beta \cdot \nabla\rho_t^\beta = \abs{\nabla\log\rho_t^\beta}^2\rho_t^\beta$ on $\supp(\rho_t^\beta)$, and using~\eqref{eq:BCs}, we obtain 
    \begin{align*}
        \frac{\dx{}}{\dx{t}}\int_{\omgg} \vartheta_t^\beta \rho_t^\beta \dx{x} &= \int_{\omge} \bra*{\frac{1}{2}\abs{\nabla\vartheta_t^\beta}^2 + \frac{\beta^2}{8}\abs{\nabla \log \rho_t^\beta}^2}\rho_t^\beta \dx{x}.
    \end{align*}
    Integrating over $[\delta, 1-\delta]$ for positive $\delta>0$ then yields
    \begin{align*}
        \int_{\omgg} \vartheta_{1-\delta}^\beta \rho_{1-\delta}^\beta \dx{x} - \int_{\omgg} \vartheta_\delta^\beta \rho_\delta^\beta \dx{x} =& \int_\delta^{1-\delta}\int_{\omge} \bra*{\frac{1}{2}\abs{\nabla\vartheta_t^\beta}^2 + \frac{\beta^2}{8}\abs{\nabla\log\rho_t^\beta}^2}\rho_t^\beta \dx{x} .
    \end{align*}
    In the term on the right-hand side, we can take the limit as $\delta\to 0$ by monotone convergence. For the left-hand side, observe that $ \vartheta_t^\beta = \psi_t^\beta - \frac{\beta}{2}\log\rho_t^\beta$, and therefore
    \begin{align}\label{eq:Equal_PosDelta}
        \int_{\omgg} \vartheta_{1-\delta}^\beta \rho_{1-\delta}^\beta \dx{x} &= \int_{\omgg} \psi_{1-\delta}^\beta\rho_{1-\delta}^\beta \dx{x} - \frac{\beta}{2}\calH(\mu_{1-\delta}^\beta \entsep\leb) .
    \end{align}
    From~\eqref{eq:ce_from_fg} together with the convergence established in Lemma~\ref{lem:RegHeat} applied to $f_t^\beta$ and $g_t^\beta$, we infer that $\psi_{1-\delta}^\beta \to \psi_1^\beta$ in $L^2(\omgg)$, and $\rho_{\delta}^\beta \to \rho_1$  in $L^2(\omgg)$, so that 
    \begin{align*}
        \int_{\omgg} \psi_{1-\delta}^\beta \rho_{1-\delta}^\beta \dx{x} \to \int_{\omgg} \psi_{1}^\beta \rho_1 \dx{x} \quad \text{and} \quad \calH(\mu_{1-\delta}\entsep\leb) \to \calH(\mu_{1}\entsep\leb).
    \end{align*}
    On the other hand, $\vartheta_t^\beta = -\varphi_t^\beta + \frac{\beta}{2}\log\rho_t^\beta$, and as before
    \begin{align*}
        \varphi_{\delta}^\beta \rho_{\delta}^\beta \dx{x} \to \int_{\omgg} \varphi_{0}^\beta \rho_0 \dx{x} \quad \text{and} \quad \calH(\mu_{\delta}\entsep\leb) \to \calH(\mu_{0}\entsep\leb). 
    \end{align*}
    Substituting these limits into~\eqref{eq:Equal_PosDelta} yields  
    \begin{align}
        \int_{\omgg} \varphi_0^\beta \rho_0 \dx{x} + \int_{\omgg} \psi_1^\beta \rho_1 \dx{x} 
        =& \frac{\beta}{2}\bra*{\calH(\mu_0\entsep\leb) + \calH(\mu_1\entsep\leb)} + \int_0^1\int_{\omge} \bra*{\frac{\abs{\nabla\vartheta_t^\beta}^2}{2} + \frac{\beta^2}{8}\abs{\nabla \log\rho_t^\beta}^2}\rho_t^\beta \dx{x}\dx{t} . \label{eq:dual_dynamic}
    \end{align}
    By definition of the entropy, it holds that
    \begin{align*}
        &\beta \calH(f^\beta\otimes g^\beta R_\beta\entsep R_\beta) = \beta \iint_{\omgg\times\omgg} (f^\beta(x) g^\beta(y)) \log(f^\beta(x)g^\beta(y))\dx{R_\beta(x,y)}\\
        =& \beta \iint_{\omgg\times\omgg} (f^\beta(x) g^\beta(y)) \log(f^\beta(x)) \dx{R_\beta(x,y)} + \beta \iint_{\omgg\times\omgg} (f^\beta(x) g^\beta(y)) \log(g^\beta(y))\dx{R_\beta(x,y)} .
    \end{align*}
    Now, the pair $(f^\beta, g^\beta)$ satisfies the Schrödinger system by Proposition~\ref{prop:existence_static_SBP}. Therefore, an application of the disintegration theorem \cite[Theorem 5.3.1]{Ambrosio_GF} and the fact that $\pi_\#^0 (f^\beta\otimes g^\beta R_\beta) = \mu_0$ allows us to rewrite the first integral as
    \begin{align*}
        \beta\iint_{\omgg\times\omgg} (f^\beta(x) g^\beta(y)) \log(f^\beta(x)) \dx{R_\beta(x,y)} &= \beta\int_{\omgg} \log(f^\beta(x)) \dx{\pi_\#^0 (f^\beta\otimes g^\beta R_\beta)(x)} = \int_{\omgg} \varphi_0^\beta(x) \dx{\mu_0(x)}.
    \end{align*}
    Similarly, it holds that $\pi_\#^1 (f^\beta\otimes g^\beta R_\beta) = \mu_1$, and the second integral is given as 
    \begin{align*}
        \beta\iint_{\omgg\times\omgg} (f^\beta(x) g^\beta(y)) \log(g^\beta(y)) \dx{R_\beta(x,y)} &= \beta\int_{\omgg} \log(g^\beta(y)) \dx{\pi_\#^1 (f^\beta\otimes g^\beta R_\beta)(y)} = \int_{\omgg} \psi_1^\beta(y) \dx{\mu_1(y)}.
    \end{align*}
    Therefore, the left-hand side of~\eqref{eq:dual_dynamic} coincides with $\beta \calH(f^\beta\otimes g^\beta R_\beta\entsep R_\beta)$, which was optimal by definition. 
    This shows that
    \begin{align*}
        \inf\limits_{\gamma\in\Gamma(\mu_0,\mu_1)} &\beta \calH(\gamma\entsep R_\beta) = \frac{\beta}{2}\bra*{\calH(\mu_0\entsep\leb) + \calH(\mu_1\entsep\leb)} + \int_0^1\int_{\omgg} \bra*{\frac{1}{2}\abs{\nabla\vartheta_t^\beta}^2 + \frac{\beta^2}{8}\abs{\nabla \log\rho_t^\beta}^2}\rho_t^\beta \dx{x}\dx{t} 
    \end{align*}
    for $\rho_t^\beta, \vartheta_t^\beta$ as constructed above. 

    \textbf{Step 2:}
    It remains to show that $(\rho^\beta_t\leb, \nabla\vartheta_t^\beta \rho_t^\beta \leb)$ is optimal in~\eqref{eq:BB_SBP}.
    To this end, we restrict to the case $(\mu_t, J_t)\in\cebound$, $\mu_t=\rho_t\leb$ and $J_t=\rho_t v_t \leb$ with
    \begin{align*}
        \int_0^1\int_{\omgg} \bra*{\frac{1}{2}\abs{v_t}^2 + \frac{\beta^2}{8}\abs{\nabla\log\rho_t}^2}\rho_t \dx{x}\dx{t} < \infty.
    \end{align*}
    Given such a pair, the family of probability measures $(\rho_t)_{t\in(0,1)}$ is absolutely continuous \cite[Theorem 3.7]{MatthiasErbar2022}. Now, our strategy is to mimic the calculations from the previous step in order to relate static and dynamic formulations. However, we can no longer rely on the additional structure of $\rho_t^\beta$ from~\eqref{eq:ce_from_fg} compared to general solutions to the continuity equation. In particular, Remark~\ref{rem:supp_fg} does not apply, and the support of $\rho_t$ may differ from the support of $\rho_t^\beta$. Therefore, the limits of the products $\varphi_\delta^\beta\rho_\delta$ and $\psi_{1-\delta}^\beta\rho_{1-\delta}$ do not follow directly as in the previous step. To circumvent this issue, we follow the ideas from \cite{gigli_benamou-brenier_2018} and, for $\ve>0$, introduce the shifted functions 
    \begin{align*}
        \varphi_t^{\beta,\ve} \coloneqq \beta \log\bra{f_t^\beta + \ve}, \quad \psi_t^{\beta,\ve} \coloneqq \beta \log\bra{g_t^\beta + \ve},  \quad \vartheta_t^{\beta,\ve} \coloneqq \bra{\psi_t^{\beta,\ve} - \varphi_t^{\beta,\ve}}/2,
    \end{align*}
    where $f_t^\beta$ and $g_t^\beta$ are still the quantities defined in~\eqref{eq:ce_from_fg}. These functions are bounded from below and have the same regularity as in Lemma~\ref{lem:Regularity_fg}. Moreover, the following equation holds pointwise 
    \begin{align*}
        \frac{\dx{}}{\dx{t}} \vartheta_t^{\beta,\ve}  + \frac{1}{4}\bra*{\beta\Delta \varphi_t^{\beta,\ve} + \beta\Delta \psi_t^{\beta,\ve} + \abs{\nabla \varphi_t^{\beta,\ve}}^2 + \abs{\nabla\psi_t^{\beta,\ve}}^2} &= 0
    \end{align*}
    as well as the boundary conditions 
    \begin{align*}
        \sum\limits_{\sfe\in\sfE(\sfv)} \nabla \varphi^{\beta,\ve,\sfe}_t(\sfv) \cdot n^{\sfe}_\sfv = 0 = \sum\limits_{\sfe\in\sfE(\sfv)} \nabla \psi^{\beta,\ve,\sfe}_t(\sfv) \cdot n^{\sfe}_\sfv.
    \end{align*}
    Note that the Sobolev space $W^{1,2}(\omgg)$ is infinitesimal Hilbertian by \cite[Theorem 3.5]{krautz2025weak}, and these properties allow us to apply the product rule from \cite[Lemma 3.5]{gigli_benamou-brenier_2018} to evaluate the following time derivative. With an additional integration by parts we conclude that 
    \begin{align*}
        \frac{\dx{}}{\dx{t}} \int_{\omgg} \vartheta_t^{\beta,\ve} \rho_t \dx{x} &= \int_{\omge} \bra*{-\frac{\abs{\nabla\psi_t^{\beta,\ve}}^2}{4} - \frac{\abs{\nabla\varphi_t^{\beta,\ve}}^2}{4} + \frac{\beta}{4}\nabla(\psi_t^{\beta,\ve} - \varphi_t^{\beta,\ve})\nabla\log \rho_t} \rho_t \dx{x} + \int_{\omge} \nabla\vartheta_t^{\beta,\ve} v_t \rho_t\dx{x}
    \end{align*}
    for a.e. $t\in (0,1)$, and Young's inequality gives 
    \begin{align*}
        \beta \nabla(\psi_t^{\beta,\ve} - \varphi_t^{\beta,\ve})\nabla\log \rho_t \leq \frac{1}{2}\abs{\nabla(\psi_t^{\beta,\ve} - \varphi_t^{\beta,\ve})}^2 + \frac{\beta^2}{2}\abs{\nabla\log \rho_t}^2 , \; \nabla\vartheta_t^{\beta,\ve} v_t \leq \frac{1}{4}\abs{\nabla(\psi_t^{\beta,\ve}-\varphi_t^{\beta,\ve})}^2 + \abs{v_t}^2.
    \end{align*}
    Similar as before, we can integrate over $[\delta,1-\delta]$ for some $\delta>0$. We obtain
    \begin{align}\label{eq:integral_delta}
        \int_{\omgg}\vartheta_{1-\delta}^{\beta,\ve} \rho_{1-\delta} \dx{x} - \int_{\omgg}\vartheta_\delta^{\beta,\ve}\rho_\delta\dx{x} &\leq \int_\delta^{1-\delta}\int_{\omgg} \bra*{\frac{\abs{v_t}^2}{2} + \frac{\beta^2}{8}\abs{\nabla\log\rho_t}^2}\rho_t \dx{x}\dx{t}.
    \end{align}
    It remains to take the limit as $\delta\to 0$, where the convergence of the right-hand side follows by monotone convergence. Using part iv) from Definition~\ref{def:CE_CEb} and Hölder's inequality, we estimate 
    \begin{align*}
        \abs*{\int_{\omgg} \vartheta_\delta^{\beta,\ve} \rho_\delta \dx{x} - \int_{\omgg} \vartheta_0^{\beta,\ve} \rho_0 \dx{x}} &\leq \abs*{\int_{\omgg} \bra*{\vartheta_\delta^{\beta,\ve} - \vartheta_0^{\beta,\ve}}\rho_\delta \dx{x}} + \abs*{\int_{\omgg} \vartheta_0^{\beta,\ve} \rho_\delta \dx{x} - \int_{\omgg} \vartheta_0^{\beta,\ve} \rho_0 \dx{x}}\\
        &\leq C\abs*{\int_{\omgg} \vartheta_\delta^{\beta,\ve} - \vartheta_0^{\beta,\ve} \dx{x}} + \abs*{\int_{\omgg} \vartheta_0^{\beta,\ve} \rho_\delta \dx{x} - \int_{\omgg} \vartheta_0^{\beta,\ve} \rho_0 \dx{x}}\\
        &\leq C' \norm{\vartheta_\delta^{\beta,\ve} - \vartheta_0^{\beta,\ve}}_{L^2(\omgg)} + \abs*{\int_{\omgg} \vartheta_0^{\beta,\ve} \rho_\delta \dx{x} - \int_{\omgg} \vartheta_0^{\beta,\ve} \rho_0 \dx{x}}
    \end{align*}
    for a constant $C'>0$. Note that the right-hand side is well-defined because of the shift by $\ve>0$. Using Lemma~\ref{lem:RegHeat}, we conclude that the first term vanishes as $\delta\to 0$. 
    Regarding the second term, note that $\vartheta_0^{\beta,\ve} \in L^\infty(\omgg)\hookrightarrow L^1(\omgg)$ for $\ve>0$, and we can approximate it by continuous functions. In particular, for any $n\in\N$ we find a function $\vartheta_0^n\in C(\omgg)$ such that $\norm{\vartheta_0^{\beta,\ve} - \vartheta_0^n}_{L^1(\omgg)} \leq \frac{1}{n}$ and $\lim_{n\to\infty} \norm{\vartheta_0^{\beta,\ve} - \vartheta_0^n}_{L^1(\omgg)} = 0$. For arbitrary $n\in\N$, we obtain
    \begin{align*}
        \abs*{\int_{\omgg} \vartheta_0^{\beta,\ve} \rho_\delta \dx{x} - \int_{\omgg} \vartheta_0^{\beta,\ve} \rho_0 \dx{x}} &\leq \int_{\omgg} \abs*{\vartheta_0^{\beta,\ve} - \vartheta_0^n} \rho_\delta \dx{x} + \abs*{\int_{\omgg} \vartheta_0^n \rho_\delta \dx{x} - \int_{\omgg} \vartheta_0^n \rho_0\dx{x}} + \int_{\omgg} \abs{\vartheta_0^n - \vartheta_0^{\beta,\ve} } \rho_0 \dx{x}\\
        &\leq \frac{2C}{n} + \abs*{\int_{\omgg} \vartheta_0^n \rho_\delta \dx{x} - \int_{\omgg} \vartheta_0^n \rho_0\dx{x}}.
    \end{align*}
    With the weak convergence, this gives in the limit
    \begin{align*}
        \lim\limits_{\delta\to 0} \abs*{\int_{\omgg} \vartheta_0^{\beta,\ve} \rho_\delta \dx{x} - \int_{\omgg} \vartheta_0^{\beta,\ve} \rho_0 \dx{x}} &= 0.
    \end{align*}
    Combining both arguments, we have shown that 
    \begin{align*}
        \lim\limits_{\delta\to0} \int_{\omgg}\vartheta_\delta^{\beta,\ve}\rho_\delta\dx{x} &= \int_{\omgg}\vartheta_0^{\beta,\ve}\rho_0\dx{x}.
    \end{align*}
    The convergence of the second term on the left-hand side of~\eqref{eq:integral_delta} follows analogously. Taking the limit $\delta\to0$ on both sides of the inequality gives
    \begin{align*}
        \int_{\omgg}\vartheta_1^{\beta,\ve} \rho_1 \dx{x} - \int_{\omgg}\vartheta_0^{\beta,\ve}\rho_0\dx{x} &\leq \int_0^1\int_{\omgg} \bra*{\frac{\abs{v_t}^2}{2} + \frac{\beta^2}{8}\abs{\nabla\log\rho_t}^2}\rho_t \dx{x}\dx{t}.
    \end{align*}
    It remains to consider the limit as $\ve\to0$ on the left-hand side. To this end, we substitute the identities 
    \begin{align*}
        \vartheta_1^{\beta,\ve} = \psi_1^{\beta,\ve} - \frac{\beta}{2}\log\bra*{\bra{f_1^\beta + \ve}\bra{g^\beta+\ve}}, \quad \vartheta_0^{\beta,\ve} = -\varphi_0^{\beta,\ve} + \frac{\beta}{2}\log\bra*{\bra{f^\beta + \ve}\bra{g_0^\beta+\ve}}
    \end{align*}
    to obtain 
    \begin{align*}
        \int_{\omgg}\vartheta_1^{\beta,\ve} \rho_1 \dx{x} - \int_{\omgg}\vartheta_0^{\beta,\ve}\rho_0\dx{x} &=
         \int_{\omgg}\psi_1^{\beta,\ve} \rho_1 \dx{x} + \int_{\omgg}\varphi_0^{\beta,\ve}\rho_0\dx{x}\\
         &\quad - \frac{\beta}{2}\int_{\omgg}\log\bra*{\bra{f_1^\beta + \ve}\bra{g^\beta+\ve}}\rho_1\dx{x}   - \frac{\beta}{2} \log\bra*{\bra{f^\beta + \ve}\bra{g_0^\beta+\ve}} \rho_0 \dx{x}\\
         &\overset{\ve\to0}{\longrightarrow} \int_{\omgg} \psi_1^\beta \rho_1\dx{x} + \int_{\omgg}\varphi_0^\beta\rho_0 \dx{x} - \frac{\beta}{2} \calH(\mu_1\entsep\leb) - \frac{\beta}{2} \calH(\mu_0\entsep\leb).
    \end{align*}
    Therefore, it holds that
    \begin{align*}
        \int_{\omgg} \psi_1^\beta \rho_1\dx{x} + \int_{\omgg}\varphi_0^\beta\rho_0 \dx{x} &\leq + \frac{\beta}{2} \calH(\mu_1\entsep\leb) + \frac{\beta}{2} \calH(\mu_0\entsep\leb)  +\int_0^1\int_{\omgg} \bra*{\frac{\abs{v_t}^2}{2} + \frac{\beta^2}{8}\abs{\nabla\log\rho_t}}\rho_t \dx{x}\dx{t},
    \end{align*}
    and as in the previous step, we can rewrite the left-hand side to get
    \begin{align*}
        \inf\limits_{\gamma\in\Gamma(\mu_0,\mu_1)} \beta \calH(\gamma\entsep R_\beta)
        &\leq \frac{\beta}{2}\bra*{\calH(\mu_0\entsep\leb) + \calH(\mu_1\entsep\leb)} + \int_0^1\int_{\omgg} \bra*{\frac{\abs{v_t}^2}{2} + \frac{\beta^2}{8}\abs{\nabla \log\rho_t}^2}\rho_t \dx{x}\dx{t} ,
    \end{align*}
    thus proving the equivalence between static and dynamic formulation.

    \textbf{Step 3:} We have already seen that $(\rho_t^\beta, \nabla\vartheta_t^\beta \rho_t^\beta)\in \cebound$ is a minimizer of~\eqref{eq:BB_SBP}. The uniqueness now follows as in \cite[Theorem 4.1]{gigli_benamou-brenier_2018}.
\end{proof}

As a consequence of the equivalence and the $\Gamma$-convergence established in Theorem~\ref{thm:gamma_cvg_static}, we obtain the convergence of the minimum values in the dynamic setting as well. 

\begin{theorem}\label{thm:gamma_cvg_dynamic}
    Let $\mu_0, \mu_1 \in \Prob(\omgg)$ with $\calH(\mu_0\entsep\leb), \calH(\mu_1\entsep \leb) < +\infty$ satisfying Assumption~\ref{as:Linfty} be given, and define
    \begin{align*}
        \calS_\beta (\mu_0,\mu_1) \coloneqq \frac{\beta}{2}&\bra*{\calH(\mu_0\entsep\leb) + \calH(\mu_1\entsep\leb)} \\&+ \inf\limits_{(\rho_t\leb,v_t\rho_t\leb)\in\cebound} \int_0^1 \int_{\omge} \bra*{\frac{1}{2} \abs{v_t}^2 + \frac{\beta^2}{8}\abs{\nabla\log \rho_t}^2 } \rho_t\dx{x}\dx{t}.
    \end{align*}
    Then, $\lim\limits_{\beta \to 0} \calS_\beta (\mu_0,\mu_1) = \frac{1}{2}W^2_2(\mu_0,\mu_1)$, where $W_2(\mu_0,\mu_1)$ is the Wasserstein distance on $\omgg$ defined in~\eqref{eq:static_OT}, and minimizers of~\eqref{eq:BB_SBP} converge to minimizers of the dynamic optimal transport problem~\eqref{eq:OT_dynamic}.
\end{theorem}

\begin{proof}
    By Theorem~\ref{thm:equiv_static_dynamic} we know that 
    \begin{align*}
        \calS_\beta(\mu_0,\mu_1) = \inf\limits_{\gamma\in\Gamma(\mu_0,\mu_1)} \beta \calH(\gamma\entsep R_\beta).
    \end{align*}
    The result is now a consequence of Theorem~\ref{thm:gamma_cvg_static} together with the Benamou-Brenier formula established in \cite[Corollary 1]{MatthiasErbar2022} and the uniqueness of the minimizers established in Theorem~\ref{thm:equiv_static_dynamic}.
\end{proof}

\subsection{Extension of the dynamic formulation}

In this section, we extend the dynamic Schrödinger problem introduced in Theorem~\ref{thm:equiv_static_dynamic} in two ways. First, we generalize the constraint set to allow for unbounded solutions to the continuity equation as well. Secondly, we include more general initial and final data $\mu_0,\mu_1\in \Prob(\omgg)$, dropping the assumption of bounded densities. 
Throughout this section, if $\mu\in\Prob(\omgg)$, $J\in \Meas(\omge)$ are absolutely continuous with respect to a reference measure we denote their densities by $\rho$ and $j$.
In order for all terms in~\eqref{eq:BB_SBP} to be well-defined, we further introduce the following function. 
Let $\psi:\R_{\geq 0}\times \R \to \R_{\geq 0}$ be given as the convex and non-negative function
\begin{align}\label{eq:LscCvxIntegrand}
    \Psi(u,v) \coloneqq \begin{cases}
        \frac{\abs{v}^2}{2u} &: u>0,\\
        0 &: u=0=v,\\
        +\infty &: \text{else}.
    \end{cases}
\end{align}
In analogy to \cite[Definition 5.2]{MatthiasErbar2022}, we introduce the Fisher information 
\begin{align}\label{eq:FisherFct}
    \calI(\mu) = \begin{cases}
        \int_{\omge} \Psi\bra*{\rho, \nabla\rho} \dx{x} &: \mu=\rho\leb \text{ and } \rho\in W^{1,1}(\omgg),\\ +\infty &:\text{else}
    \end{cases},
\end{align}
and the action functional 
\begin{align}\label{eq:ActionFct}
    \calA_{\sfm}(\mu, J) := \begin{cases}
        \int_{\omge} \Psi\bra{\rho, j} \dx{\sfm} &: \mu=\rho\sfm \text{ and } J=j\sfm\\ +\infty &:\text{else}
    \end{cases}
\end{align}
for a reference measure $\sfm\in\MeasPos(\omgg)$.
By \cite[Theorem 3.3]{bouchitte_new_1990}, the action functional is lower semi-continuous with respect to weak convergence of measures.
Additionally, it is $1$-homogeneous, making it independent of the choice of reference measure $\sfm\in\MeasPos(\omgg)$. If $\sfm=\leb$ we shorten the notation and write $\calA = \calA_\leb$ instead.
For $\beta>0$ this leads to the extended formulation of~\eqref{eq:BB_SBP}
\begin{align}\label{eq:SBPunreg}
    \inf\limits_{(\mu_t, J_t)\in \ceunreg} \calJ_\beta(\mu_t, J_t) \coloneqq \inf\limits_{(\mu_t, J_t)\in \ceunreg} \int_0^1 \calA(\mu_t, J_t) \dx{t} + \frac{\beta^2}{4} \int_0^1 \calI(\mu_t) \dx{t}.
\end{align}
Further, we define 
\begin{align*}
     \scrA(\mu_t, J_t) \coloneqq \int_0^1 \calA_\sfm(\mu_t, J_t) \dx{t},
\end{align*}
and if $\abs{J_t}\ll\mu_t$ for almost all $t\in(0,1)$ with density $\frac{\dx{J_t}}{\dx{\mu_t}} =: v_t$, the reference measure $\sfm = \mu_t$ leads to
\begin{align*}
    \int_0^1 \int_{\omge} \frac{1}{2}\abs*{v_t}^2 \dx{\mu_t} \dx{t}.
\end{align*}
Thus, we recover~\eqref{eq:OT_dynamic}.
If $\calJ_\beta(\mu_t,J_t)<+\infty$ for $(\mu_t,J_t)\in\ceunreg$, then $\mu_t$ enjoys additional regularity because of the finiteness of the Fisher information. For this reason, it will be convenient to define the set
\begin{align*}
    \cereg := \set*{(\mu_t,J_t)\in\ceunreg \setsep \mu_t, \abs{J_t}\ll\leb, \; \rho_t\in W^{1,1}(\omgg) \text{ for a.e. } t\in(0,1)},
\end{align*}
as well as the restricted problem 
\begin{align}\label{eq:SBPreg}
    \inf\limits_{(\mu_t,J_t)\in\cereg} \calJ_\beta(\mu_t,J_t).
\end{align}
However, we need to ensure that elements in $\cereg$ exist without imposing Assumption~\ref{as:Linfty}. 
In order to construct such curves, we again rely on properties of the heat kernel on metric graphs. In particular, we make use of its eigenfunction expansion~\eqref{eq:eigenexp}.
As the kernel satisfies the heat equation pointwise, see \cite{roth_spectre_1984}, its restriction to each edge defines a solution to the one-dimensional heat equation under suitable boundary conditions. Therefore, we can choose the eigenfunctions on the graph to be suitable eigenfunctions of the one-dimensional Laplacian. This leads to the ansatz
\begin{align}\label{eq:easy_expansion}
    \psi_k^\sfe (x) = a^\sfe_k \cos(\sqrt{\lambda_k} x) + b^\sfe_k \sin(\sqrt{\lambda_k} x)
\end{align}
for coefficients $a^\sfe_k, b^\sfe_k \geq 0$ to be determined depending on the one-dimensional boundary conditions.

Motivated by the discussion from Subsection~\ref{sec:discussion}, we define 
\begin{align}\label{eq:DefRecSeq}
    \mu_t^\beta := H_{\beta h(t)} \mu_t 
\end{align}
for $h(t) := \min\set{t,1-t}$. Again, we formally evaluate the time-derivative as in~\eqref{eq:formal_dt}. Here, we do not aim for sharp action bounds but rather rewriting the time derivative as a divergence of a suitable flux. To this end, we introduce a second kernel.

\begin{lemma}\label{NewKernel}
    Let $h_t(x,y)$ be the heat kernel from Definition~\ref{kernelPath} with eigenvalues $0 =\lambda_1 < \lambda_2 \leq \ldots$ and an orthonormal basis of generalized eigenfunctions $\bra*{\psi_k}_{k\in\N}$. For the functions $\Tilde{\psi}^\sfe_k(x) := - a_k \sin(\sqrt{\lambda_k}x) + b^\sfe_k \cos(\sqrt{\lambda_k}x) = \frac{1}{\sqrt{\lambda_k}}\nabla {\psi^\sfe_k}(x)$ the kernel
    \begin{align*}
        \Tilde{h}_t(x,y) &:= \sum\limits_{k=2}^\infty e^{-\lambda_k t} \Tilde{\psi}_k(x) \Tilde{\psi}_k(y)
    \end{align*}
    is well defined. Moreover, it holds that $-\Delta \Tilde{\psi}^\sfe_k = \lambda_k \Tilde{\psi}^\sfe_k$ on each edge $\sfe\in\sfE$ with the coupling conditions 
    \begin{align*}
        \sum\limits_{\sfe\in\sfE(v)} \Tilde{\psi}_k^\sfe(v)\cdot n^\sfe_\sfv = 0 \quad \text{and} \quad \nabla\Tilde{\psi}^\sfe_k(v) = \nabla\Tilde{\psi}^f_k(v).
    \end{align*}
\end{lemma}
\begin{proof}
    Note that the sum starts at $k=2$, so that $\lambda_k>0$ is invertible. Moreover, by definition $-\Delta\psi_k(x) = \lambda_k \psi_k(x)$ in $C^\infty(\omge)$. Taking the derivative and dividing by $\sqrt{\lambda_k}$ shows that $\Tilde{\psi}_k$ solves the same equation, however, with different coupling conditions in the vertices. In particular, for all $\sfv\in\sfV$ we have that
    \begin{align*}
        0 = \sum\limits_{\sfe\in\sfE(v)} \nabla\psi_k^\sfe(v)\cdot n^\sfe_\sfv = \sqrt{\lambda_k} \sum\limits_{\sfe\in\sfE(v)} \Tilde{\psi}_k^\sfe(v)\cdot n^\sfe_\sfv,
    \end{align*}
    implying $\sum\limits_{\sfe\in\sfE(v)} \Tilde{\psi}_k^\sfe(v)\cdot n^\sfe_\sfv = 0$ by definition. Further
    \begin{align*}
        \lambda_k \psi_k(v) = \lambda_k\psi_k^\sfe(v) = - \Delta\psi_k^\sfe = - \sqrt{\lambda_k} \nabla \Tilde{\psi}^\sfe_k(v)
    \end{align*}
    for all $\sfe\in\sfE(v)$ which gives $\nabla \Tilde{\psi}_k^\sfe(v) = -\sqrt{\lambda_k}\psi_k(v)$ independent of the edge. 
    As noted in \cite[Section 11.2]{kurasov_spectral_2024}, there exists a constant $C>0$ such that $\abs{\psi_k(x)} \leq C \norm{\psi_k}_{L^2(\omgg)} = C$. Therefore we find a uniform bound $\Tilde{C}>0$ on the coefficients $a_k$, $b_k$. 
    As $\frac{1}{\sqrt{\lambda_k}} \leq \frac{1}{\sqrt{\lambda_2}}$ for $k\geq 2$ we conclude that
    \begin{align*}
        \sum\limits_{k=1}^\infty \abs*{e^{-\lambda_k t} \Tilde{\psi}^\sfe_k(x) \Tilde{\psi}^\sfe_k(y)} \leq \frac{\Tilde{C}^2}{\lambda_2} \sum\limits_{k=1}^\infty e^{-\lambda_k t}
    \end{align*}
    and the right hand side converges as shown in \cite{roth_spectre_1984}.
    
\end{proof}

With this new kernel we can rewrite~\eqref{eq:formal_dt} to show the following.

\begin{proposition}\label{prop:nonempty}
    For $\mu_i \in \calP(\omgg)$ with $\calH(\mu_i\entsep\leb)<+\infty$, $i=0,1$, we have $\ceunreg\neq\emptyset$ as well as $\cereg\neq\emptyset$. 
\end{proposition}

\begin{proof}
    The statement $\ceunreg\neq \emptyset$ follows as the Wasserstein space on metric graphs is a geodesic space, together with the Benamou-Brenier formulation established in \cite[Corollary 1]{MatthiasErbar2022}. 
    It remains to show that $\cereg\neq\emptyset$.
    To this end, fix a curve $(\mu_t, J_t)\in \ceunreg$ with finite action, which exists due to the previous considerations. Then, as the entropy along the heat flow is finite, we have $\mu_t^\beta \ll \leb$ for a.e. and $\rho_t^\beta \in W^{1,1}(\omgg)$ by the regularity properties of the heat-kernel. 
    Let $\phi\in C^1([0,1])\cap C^1(\omgg)$ be a test-function, where the space is defined in analogy to~\eqref{eq:Notation_C1Cinfty}.
    Applying Fatou's lemma and substituting $\frac{\dx{}}{\dx{t}}(\phi_t h_{\beta h(t)}) = \partial_t \phi_t h_{\beta h(t)} + \phi_t \partial_t h_{\beta h(t)} \beta h'(t)$ gives 
    \begin{align*}
        &\int_0^1 \int_{\omgg} \partial_t \phi_t(x) \dx{\mu_t^\beta(x)} \dx{t} = \int_{\omgg} \int_0^1 \int_{\omgg}  \partial_t \phi_t(x) h_{\beta h(t)}(x,y) \dx{\mu_t(y)} \dx{t}\dx{x}\\
        &=  \int_{\omgg}  \int_0^1 \int_{\omgg} \frac{\dx{}}{\dx{t}}\bra*{\phi_t(x) h_{\beta h(t)}(d_\sfG(x,y))} \dx{\mu_t(y)} \dx{t} \dx{x}  - \int_{\omgg}  \int_0^1\int_{\omgg} \phi_t(x) \beta h'(t) \partial_t h_{\beta h(t)}(x,y) \dx{\mu_t(y)} \dx{t} \dx{x}.
    \end{align*}
    Since $(\mu_t, J_t)\in\ceunreg$, and the heat kernel solves the heat equation pointwise with standard coupling conditions, integration by parts yields 
    \begin{align}
        \int_0^1 \int_{\omgg} \partial_t \phi_t(x) \dx{\mu_t^\beta(x)} \dx{t} = -\int_0^1 &\int_{\omge} \nabla_y \bra*{\int_{\omgg}\phi_t(x) h_{\beta h(t)}(x,y) \dx{x}} \dx{J_t(y)} \dx{t} \notag \\&-  \int_0^1 \int_{\omge} \phi_t(x) \beta h'(t) \int_{\omgg} \Delta_x h_{\beta h(t)}(d_\sfG(x,y))\dx{x} \,  \dx{\mu_t(y)}  \dx{x} \dx{t} \notag\\
        = -\int_0^1 &\int_{\omge} \nabla_y \bra*{\int_{\omgg}\phi_t(x) h_{\beta h(t)}(x,y) \dx{x}} \dx{J_t(y)} \dx{t} \label{eq:DistrDeriv}\\& +  \int_0^1 \int_{\omge} \nabla_x\phi_t(x) \beta h'(t) \nabla_x \bra*{\int_{\omgg} h_{\beta h(t)}(d_\sfG(x,y))\dx{x} \,  \dx{\mu_t(y)}}  \dx{x} \dx{t}.\notag
    \end{align}
    Further, the new kernel $\Tilde{h}_t(x,y)$ defined in Lemma~\ref{NewKernel} satisfies
    \begin{align*}
        \nabla_x \Tilde{h}^\sfe_t(x,y)
        &= \sum\limits_{k=1}^\infty -\sqrt{\lambda_k} e^{-\lambda_k t} \bra*{a_k \sin(\sqrt{\lambda_k}x) + b^\sfe_k \cos(\sqrt{\lambda_k}x)}\cdot \bra*{-a_k \sin(\sqrt{\lambda_k}y) + b^\sfe_k \cos(\sqrt{\lambda_k}y)},
        \intertext{and similarly for the standard kernel by~\eqref{eq:easy_expansion}}
        \nabla_y h^\sfe_t(x,y) 
        &= \sum\limits_{k=1}^\infty \sqrt{\lambda_k} e^{-\lambda_k t} \bra*{a_k \cos(\sqrt{\lambda_k}x) + b^\sfe_k \sin(\sqrt{\lambda_k}x)}\cdot \bra*{-a_k \sin(\sqrt{\lambda_k}y) + b^\sfe_k \cos(\sqrt{\lambda_k}y)}.
    \end{align*}
    Comparing both identities, we observe that
    \begin{align*}
        \nabla_y h_t(x,y) = - \nabla_x \Tilde{h}_t(x,y)
    \end{align*}
    on each edge.
    Substituting this relation into~\eqref{eq:DistrDeriv} and integrating by parts then gives
    \begin{align*}
        &\int_0^1 \int_{\omgg} \partial_t \phi_t(x) \dx{\mu_t^\beta(x)} \dx{t}\\
        =& \int_0^1\int_{\omgg}  \int_{\omge}  \phi_t(x) \nabla_x\Tilde{h}_{\beta h(t)}(d_\sfG(x,y))\dx{x} \, \dx{J_t(y)} \dx{t} +  \int_0^1 \int_{\omge} \nabla_x\phi_t(x) \beta h'(t) \nabla_x \mu_t^\beta (x) \dx{x} \dx{t}\\
        =& -\int_0^1\int_{\omge}  \nabla_x \phi_t(x)  \int_{\omge} \Tilde{h}_{\beta h(t)}(x,y) \dx{J_t(y)} \dx{x} \dx{t}  +  \int_0^1 \int_{\omge} \nabla_x\phi_t(x) \beta h'(t) \nabla_x \mu_t^\beta (x) \dx{x} \dx{t}. 
    \end{align*}
    Now, with the definition
    \begin{align*}
        J_t^\beta (x) := \int_{\omge} \Tilde{h}_{\beta h(t)}(x,y) J_t(y)\dx{y} - \beta h'(t) \nabla \mu_t^\beta (x) ,
    \end{align*}
    the pair $(\mu_t^\beta, J_t^\beta)_{t\in[0,1]}$ is a weak solution to the continuity equation. Together with Lemma~\ref{lem:RegHeat}, this shows $(\mu_t^\beta, J_t^\beta)\in\cereg$.
\end{proof}

Still,~\eqref{eq:SBPreg} might not be finite if no solution with finite Fisher information exists. However, if such a curve does exist, the problem admits a minimizer.
To prove this result, we apply the direct method to a slightly modified version of~\eqref{eq:SBPreg} which coincides with the initial formulation on $\cereg$, and therefore with~\eqref{eq:SBPunreg} as well.
In general, the additional weak differentiability is not preserved along weakly converging sequences. 
For this reason, we introduce a third measure as a substitute for the weak gradient and extend the functional, thus allowing us to consider the new measure independently of $\mu_t$.
We follow the construction of such extensions from \cite{Dolbeaut_2008}.
Let $\sfm\in\MeasPos(\omge)$ be a reference measure.
By Lebesgue's decomposition theorem, any $\nu \in \calM(\omge)$ has a unique decomposition $\nu = \nu^a + \nu^s$, where $\nu^a\ll\sfm$ and $\nu^s\perp\sfm$. Moreover, we find $\sfm^\perp\in\calM_{\geq0}(\omge)$ such that $\sfm^\perp \perp \sfm$ and $\nu^s\ll\sfm^\perp$. We define 
\begin{align*}
    \Tilde{\calJ_\beta} (\mu_t, \nu_t, J_t) = \int_0^1 \calA_\sfm\Big(\frac{\dx{\mu^a_t}}{\dx{\sfm}},& \frac{\dx{J^a_t}}{\dx{\sfm}}\Big) + \calA_{\sfm^\perp}\Big(\frac{\dx{\mu^s_t}}{\dx{\sfm^\perp}}, \frac{\dx{J^s_t}}{\dx{\sfm^\perp}}\Big)\dx{t} \\
    &+\frac{\beta^2}{8} \int_0^1\calA_\sfm\Big(\frac{\dx{\mu^a_t}}{\dx{\sfm}}, \frac{\dx{\nu^a_t}}{\dx{\sfm}}\Big) + \calA_{\sfm^\perp}\Big(\frac{\dx{\mu^s_t}}{\dx{\sfm^\perp}}, \frac{\dx{\nu^s_t}}{\dx{\sfm^\perp}}\Big)\dx{t}
\end{align*}
as a functional on $\calM_{\geq0}([0,1]\times\omgg)\times\calM([0,1]\times\omge)\times\calM([0,1]\times\omge)$.

\begin{remark}\label{rem:Prop_tildeJ}
The extended functional $\Tilde{\calJ}_\beta$ enjoys several properties.
\begin{enumerate}[label=\roman*)]
    \item\label{samefct} It holds that $\Tilde{\calJ_\beta} (\mu_t, \nu_t, J_t) = \calJ_\beta(\mu_t,J_t)$ for $(\mu_t,J_t)\in\cereg$ with $\frac{\dx{\nu_t}}{\dx{\leb}} = \nabla \rho_t$ and $\mu_t, \abs*{J_t}\ll \leb$. In this case, the singular parts are zero and we can neglect the second integral. 
    \item As the integrand is $1$-homogeneous, the definitions of $\calJ_\beta$ and $\Tilde{\calJ_\beta}$ are independent of the choice of reference measure.
    \item By \cite[Theorem 2.1]{Dolbeaut_2008} the functional $\Tilde{\calJ_\beta}$ is lower semi-continuous with respect to weak convergence of measures. \label{lscextfct}
\end{enumerate}
\end{remark}

Property iii) from Remark~\ref{rem:Prop_tildeJ} establishes the first key ingredient of the direct method, namely lower semicontinuity. The second ingredient (compactness) needs a more detailed analysis.

\begin{theorem}\label{CmpctSub}
    For $\Lambda>0$, the set $M_\Lambda := \set*{(\mu_t,J_t)\in\cereg \setsep\calJ_\beta(\mu_t, J_t) \leq \Lambda}$ is compact with respect to weak convergence of measures.
\end{theorem}

\begin{proof}
   Let $(\mu_t^n, J_t^n)_{n\in\N}\subset M_\Lambda$ be given. 
   The proof is divided into two steps. In the first part, we show compactness separately for each of the sequences $(\mu_t^n)_{n\in\N}$, $(\nu_t^n)_{n\in\N}$, and $(J_t^n)_{n\in\N}$. Here, $\nu_t^n$ is the measure defined by the density $\nabla\rho_t^n$ against the Lebesgue measure on $\omge$ and by i) from Remark~\ref{rem:Prop_tildeJ} we have $\Tilde{\calJ_\beta}(\mu_t^n,\nu_t^n,J_t^n) = \calJ_\beta(\mu_t^n,J_t^n) \leq \Lambda$ as well.
   Next, we conclude that the limit is contained in the set $M_\Lambda$.
   
   For the first part, we apply \cite[Lemma 5.1]{Stephan2021} to the choices $\Omega=[0,1]\times\omge$, $F(r) = r^2$, $\mu = \mu_t^n$, and $W=\nabla\mu_t^n $ or $W=J_t^n$ to obtain
   \begin{align} \label{eq:grad_bound}
       \norm*{\nu^n} = \norm*{\nabla \rho_t^n}_{\spaceflux} &\leq \frac{8\Lambda}{\beta^2} + k_F \quad \text{and}\quad \norm*{J^n} = \norm*{j_t^n}_{\spaceflux} \leq 2\Lambda + k_F
   \end{align}
   for all $n\in\N$. As a consequence, we find weakly converging subsequences
   \begin{align*}
       \nu^n \weakcvg \nu \text{ in } \calM_{\geq0}([0,1]\times\omge)\quad \text{and} \quad J^n \weakcvg J \text{ in } \calM_{\geq0}([0,1]\times\omge).
   \end{align*}
   Moreover, by standard arguments, see e.g. \cite[Lemma A.2]{bredies2020optimal}, the sequence $(\mu_t^n)_{n\in\N}$ satisfies the Hölder estimate
    \begin{align*}
       \abs*{\int_{\omgg} \phi(x) \dx{\mu_t^n(x)} - \int_{\omgg} \phi(x) \dx{\mu_s^n(x)} \dx{x}} &\leq C \sqrt{\Lambda} \norm{\phi}_{C^1(\omge)}\abs{t-s}^\frac{1}{2}
   \end{align*}
   for a constant $C>0$. A generalized version of Arzelà-Ascoli \cite[Lemma A.4]{bredies2020optimal} then gives $\mu^n \weakcvg \mu$ in $\calM_{\geq0}([0,1]\times\omgg)$ and $\mu_t^n \weakcvg\mu_t$ in $\Prob(\omgg)$ for all $t\in(0,1)$, as well as weak continuity of $t\mapsto \mu_t$. 
   Regarding the disintegration in time of the limit measures $\nu^n$ we obtain from Hölder's and Jensen's inequalities that
   \begin{align*}
       \int_\omge \abs{\nabla\rho_t^n} \dx{x} &\leq \int_{\omge\cap \set{\rho_t^n\neq0}} \sqrt{\Psi\bra*{\rho_t^n(x), \nabla\rho_t^n(x)}}\sqrt{\rho_t^n(x)} \dx{x} + \int_{\omge\cap \set{\rho_t^n=0}} \sqrt{\Psi\bra*{\rho_t^n(x), \nabla\rho_t^n(x)}} \dx{x}\\&\leq C \calI(\mu_t^n)^\frac{1}{2}
   \end{align*}
   for $n\in\N$ , almost all $t\in(0,1)$, and a constant $C>0$. From this bound, the assumptions of the disintegration theorem \cite[Theorem 5.3.1]{Ambrosio_GF} follow and we obtain $\nu_t^n \weakcvg\nu_t$ for almost all $t\in(0,1)$ as well as $\nu = (\nu_t)_{t\in[0,1]}$. By analogous computations, it holds that $J_t^n\weakcvg J_t$ for almost all $t\in(0,1)$ with $J = (J_t)_{t\in[0,1]}$. 
   As the continuity equation is linear, we conclude $(\mu_t,J_t)\in\ceunreg$.
   Next, we consider the absolute continuity of $\mu_t$ with respect to $\leb$.   
   Since $\leb$ is outer regular, we can find a sequence of continuous functions $\varphi_k:\omgg\to [0,1]$, $k\in\N$, such that $\varphi_k\geq \one_A$ and $\varphi_k\to \one_A$ in $L^1(\omgg)$ as $k\to\infty$. For any $k\in\N$, we obtain 
   \begin{align*}
       0 &\leq \int_0^1 \mu_t(A) \dx{t} = \int_0^1 \int_{\omgg} \one_A(x) \dx{\mu_t(x)}\dx{t}\leq \int_0^1 \int_{\omgg} \varphi_k(x) \dx{\mu_t(x)}\dx{t} = \lim\limits_{n\to\infty} \int_0^1 \int_{\omgg} \varphi_k(x) \dx{\mu^n_t(x)}\dx{t}.
   \end{align*}
   Using the fact that $W^{1,1}(\omgg)\hookrightarrow C(\omgg)$ \cite[Lemma 3.27]{mugnolo_semigroup_2014} with embedding constant $C_{W^{1,1}}>0$ together with~\eqref{eq:grad_bound}, we conclude
   \begin{align*}
       0 &\leq \int_0^1 \int_{\omgg} \varphi_k(x) \dx{\mu^n_t(x)}\dx{t} \leq C_{W^{1,1}}\norm{\varphi_k}_{L^1(\omgg)} \int_0^1\norm*{\rho_t^n}_{W^{1,1}(\omgg)}\dx{t} \leq C_{ac} \norm{\varphi_k}_{L^1(\omgg)} 
   \end{align*}
   for a constant $C_{ac}>0$. As $k\in\N$ was arbitrary and $\varphi_k\to \one_A$ in $L^1(\omgg)$, this gives
   \begin{align*}
       0 &= \int_0^1 \mu_t(A)\dx{t}
   \end{align*}
   allowing us to apply the Radon-Nikodym theorem, and consequently $\mu_t\ll\leb$ for a.e. $t\in[0,1]$.
   
    It remains to show that the limit is an element of the set $M_\Lambda$. We have $\Tilde{\calJ_\beta}(\mu_t, J_t, \nu_t) \leq \Lambda$ by lower semicontinuity of $\Tilde{\calJ}_\beta$, and because of $\mu_t^s \equiv 0$, this implies $\nu_t^s \equiv 0$ as well as $J_t^s\equiv 0$. As a consequence, $\nu_t\ll\leb$ and by weak convergence together with the linearity of the weak gradient we conclude $\rho_t \in W^{1,1}(\omgg)$ for a.e. $t\in(0,1)$, and therefore $(\mu_t,J_t)\in\cereg$.  
   On this set, $\Tilde{\calJ}_\beta$ is lower semicontinuous and coincides with $\calJ_\beta$, so that $(\mu_t,J_t)\in M_\Lambda$, thus proving compactness.
\end{proof}

Having established lower semi-continuity and compactness, we can apply the direct method leading to the following existence result. 

\begin{theorem}
    Let $\mu_0, \mu_1 \in \calP(\omgg)$ with $\calH(\mu_i\entsep\leb)<+\infty$, $i=0,1$ be given. If the infimum is finite, there exists $(\mu_t^*, J_t^*)\in\cereg$ minimizing~\eqref{eq:SBPreg} and~\eqref{eq:SBPunreg}. 
\end{theorem}

\section{Numerics}\label{numerics}

In this section, we numerically study~\eqref{eq:SBPunreg} on metric graphs and present several examples. 
Our algorithm is based on a primal-dual approach from \cite{carrillo2021primal} and \cite{Pietschmann_2022}, modified by introducing a mixed formulation. In turn, this leads to slightly different proximal operators. Compared to \cite{Pietschmann_2022}, we employ central difference quotients for time and space derivatives in order to retain the time-symmetric structure and as in \cite{Pietschmann_2022}, we enforce the constraint set up to a given tolerance level. This approach results in an unconstrained minimization problem.

Throughout we assume that there exists a tuple $(\mu_t, J_t)\in \cereg$ such that $\calJ_\beta(\mu_t,J_t)< +\infty$, allowing us to consider~\eqref{eq:SBPreg} instead of~\eqref{eq:SBPunreg}. This is true in particular under Assumption~\ref{as:Linfty}, which also allows to consider absolute continuous minimizers with respect to $\leb$. For this reason, we consider the densities $(\rho_t, j_t)$ instead of the measures $(\mu_t,J_t)$.

\subsection{Discretization}

We discretize each edge $\sfe\in\sfE$ associated to the closed interval $[0,\ell_\sfe]$ for $\ell_\sfe>0$ using an equidistant grid consisting of $N_x^\sfe$ subintervals of length $\Delta_x^\sfe:= \ell_\sfe/N_x^\sfe$. Similarly, the time interval $[0,1]$ is discretized using $N_t$ intervals of length $\Delta_t:=1/N_t$. The corresponding gridpoints are denoted by 
\begin{align*}
x^\sfe_n = (n-1)\Delta^\sfe_x \quad \text{and} \quad t_k = (k-1)\Delta_t
\end{align*}
for $\sfe\in\sfE$, $n\in\{1,\ldots,N^\sfe_x+1\}$ and $k\in\{1,\ldots,N_t+1\}$. To any continuous function $v_t:[0,1]\times[0,\ell_\sfe]\to\R$ on this grid we associate the gridfunction $v_h:=(v^\sfe_{n,k})_{\sfe,n,k}$ defined by the evaluations 
\begin{align*}
v^\sfe_{n,k} := v^\sfe(x_n^\sfe, t_k).
\end{align*}
Note that this definition leads to different vertex values on each edge. 
In order to adapt the primal-dual approach, we introduce the additional variable
\begin{align*}
g^\sfe_t := -\frac\beta2\nabla\rho^\sfe_t
\end{align*}
and the resulting mixed formulation
\begin{align*}
\min\limits_{(\rho_t, j_t, g_t)\in\ceregmix}  \int\limits_0^1 \calA(\rho_t, j_t) + \calA(\rho_t, g_t) \dx{t},
\end{align*}
where $\ceregmix$ is the set of weak solutions to the system 
\begin{equation}\label{eq:MixedForm}\tag{MF}
\begin{cases}
     \partial_t\rho_t + \nabla\cdot j_t = 0  &(0,1)\times\omge\\
     g_t = -\frac\beta2\nabla\rho_t &(0,1)\times\omge\\
     \sum_{\sfe\in\sfE(\sfv)} j_t(\sfv) n^\sfe_\sfv = 0 &\sfv\in\sfV\\
     \rho^\sfe_t(\sfv) = \rho^{\sff}_t(\sfv) &\sfv\in\sfV, \, \sfe,\sff\in\sfE(\sfv)\\
     \rho_{|t=0} = \rho_0 &\omgg\\
     \rho_{|t=1} = \rho_1 &\omgg
\end{cases}
\end{equation}
and $\calA(\cdot,\cdot)$ is as in~\eqref{eq:ActionFct}. To shorten notation, we introduce the following shorthand of the integrand for this mixed formulation 
\begin{align*}
    \Psi^{\rm mix}(\rho_t,j_t,g_t) &= \Psi(\rho_t,j_t) + \Psi(\rho_t,g_t).
\end{align*}

Next, we discretize the constraint set. Regarding the continuity equation, we employ central difference quotients in the interior points and one sided difference quotients in the vertices. They are defined by the discrete differential operators
\begin{align*}
D^t\rho^\sfe_h(x^\sfe_n, t_k) := \begin{cases}
        \frac{\rho^\sfe_{n,2} - \rho^\sfe_{n,1}}{\delta_t}        &: k=1 \\
        \frac{\rho^\sfe_{n,k+1} - \rho^\sfe_{n,k-1}}{2\delta_t}   &: 2\leq k \leq N_t \\
        \frac{\rho^\sfe_{n,N_t+1} - \rho^\sfe_{n,N_t}}{\delta_t}        &: k=N_t+1 
\end{cases}
\; \text{ and } \;
D^x j^\sfe_h (x^\sfe_n, t_k) := \begin{cases}
        \frac{j^\sfe_{2,k} - j^\sfe_{1,k}}{\delta^\sfe_x}        &: n=1 \\
        \frac{j^\sfe_{n+1,k} - j^\sfe_{n-1,k}}{2\delta^\sfe_x}   &: 2\leq n \leq N^\sfe_x \\
        \frac{j^\sfe_{N^\sfe_x+1,k} - j^\sfe_{N^\sfe_x,k}}{\delta^\sfe_x}        &: n=N^\sfe_x+1 
    \end{cases}
\end{align*}
acting on gridfunctions $\rho_h$ and $j_h$. Integration is approximated using the composite trapezoidal rule with weights 
\begin{align*}
\omega^\sfe_n := \begin{cases}
    \frac{\Delta^\sfe_x}{2} &: n\in\{1, N^\sfe_x+1\} \\
    \Delta^\sfe_x &: \text{ else}
\end{cases}
\quad \text{and} \quad
\omega_k := \begin{cases}
    \frac{\Delta_t}{2} &: k\in\{1, N_t+1\} \\
    \Delta_t &: \text{ else}
\end{cases},
\end{align*}
which allows us to define the discrete inner product 
\begin{align*}
\skp{\bm{u}_h,\bm{v}_h} := \sum\limits_{\sfe\in\sfE} \sum\limits_{n=1}^{N^\sfe_x+1}\sum\limits_{k=1}^{N_t+1} \omega^\sfe_n \omega_k \left(\rho^{u,e}_{n,k}\rho^{v,e}_{n,k} + j^{u,e}_{n,k}j^{v,e}_{n,k} + g^{u,e}_{n,k}g^{v,e}_{n,k} \right),
\end{align*}
where $u_h = \left(\rho^u_h, j^u_h, g^u_h \right)$ and $v_h = \left(\rho^v_h, j^v_h, g^v_h \right)$. We denote the induced norm by $\norm{u_h} = \sqrt{\skp{u_h, u_h}}$.

\subsection{Algorithm}

In the following, we introduce the basic ideas and steps of the algorithm, which is based on the works \cite{carrillo2021primal} and \cite{Pietschmann_2022}. 

Due to the numerical discretization, we add the conservation of mass as an additional constraint. Each of the seven constraints is enforced up to a given tolerance $\delta_i>0$ for $i\in\set{1,\ldots,7}$, using the norm introduced in the previous section. For example, the continuity equation is enforced by
\begin{align*}
\sum\limits_{\sfe\in\sfE} \sum\limits_{n=1}^{N^\sfe_x+1} \sum\limits_{k=1}^{N_t+1} \omega^\sfe_x\omega_t \left(D^t\rho^\sfe_{n,k} + D^xj^\sfe_{n,k} + g^\sfe_{n,k}\right)^2  \leq \delta_1^2.
\end{align*}
In this relaxed form, we are dealing with quadratic terms that can be rewritten as
\begin{align*}
S u_h\in C_\delta := \left\{ x \ | \ \Vert x_i - b_i \Vert_2 \leq \delta_i \text{ for } i=1,\ldots,7 \right\}
\end{align*}
for $S$ denoting the matrix obtained from the discretization of the left hand side of each constraint and $b$ the vector corresponding to the discretized right hand sides. With $\iota_\delta$ the convex indicator function of $C_\delta$, we can rewrite the optimization problem as the following unconstrained one 
\begin{align*}
\inf\limits_{u_h = \left(\rho_h, g_h, j_h\right)} \sum\limits_{\sfe\in\sfE} \sum\limits_{n=1}^{N^\sfe_x+1}\sum\limits_{k=1}^{N_t+1} \omega^\sfe_n \omega_k \Psi^{\rm mix}\left(\rho^\sfe_{n,k}, j^\sfe_{n,k}, g^\sfe_{n,k} \right) + \iota_\delta\left( S u_h \right).
\end{align*}
This problem can be solved by the primal-dual algorithm~\ref{Alg1}, where we make use of the Legendre-Fenchel-transform of the convex indicator function
\begin{align*}
\iota_\delta^*(\Tilde{u}_h) := \max\limits_{u_h} ~ \skp{\Tilde{u}_h, u_h}_2 - i_\delta(u_h)
\end{align*}
and  the proximal operator of a proper, convex and lower semicontinuous map $\Psi$, defined as
\begin{align*}
\prox_\Psi(u) := \argmin\limits_{v} \Psi(v) + \frac{1}{2}\norm{v - u}^2.
\end{align*}

\begin{algorithm}[t]
    \caption{primal-dual algorithm for the Schrödinger problem}\label{Alg1}
    \SetKwInOut{Input}{Input}
    \SetKwInOut{Output}{Output}
    \Input{$u^{(0)}$, $\phi^{(0)}$, \text{tol}, $l_{\max}$, $\lambda$, $\sigma$, $S$, $b$, $\delta$}
    \Output{$u_{\text{opt}}$, $\phi_{\text{opt}}$}

    Initialize $\Bar{u}^{(0)} = u^{(0)}$ \\
    \For{$l\in\{0,\ldots, l_{\max}\}$}
    {
        \begin{flushleft}
        $\phi^{(l+1)} = \prox_{\sigma \iota_\delta^*}\left( \phi^{(l)} + \sigma S\Bar{u}^{(l)}\right)$ \newline
        $u^{(l+1)} = \prox_{\lambda \Psi^{\rm mix}}\left( u^{(l)} - \lambda S^* \phi^{(l+1)}\right)$ \newline
        $\overline{u}^{(l+1)} = 2 u^{(l+1)} - u^{(l)}$ 
        
        \If{$\Vert Su - b\Vert_2<\text{tol}$}
        {
            $u_{\text{opt}}= u^{(l+1)}$         
        }
        \end{flushleft}
    }
\end{algorithm}

It remains to evaluate the proximal mappings. As in \cite{carrillo2021primal}, the proximal operator of the indicator function coincides with the projection onto the set $C_\delta$. 
For the second operator, we need to solve 
\begin{align*}
\prox_{\lambda \Psi^{\rm mix}}(u) = \argmin\limits_{v = (\rho^v, g^v, j^v)} \lambda\Psi(\rho^v, j^v) + \lambda\Psi(\rho^v, g^v) + \frac{1}{2}\norm{v-u}^2
\end{align*}
for $u=(\rho^u, g^u, j^u)$ . 
Note that $\lambda\Psi(\rho^v, j^v) + \lambda\Psi(\rho^v, g^v) + \frac{1}{2}\norm{v-u}^2$ is strictly convex and non-negative for all $u$ with $\rho^u\geq 0$. Therefore, minimizers are unique and can be characterized as roots of the gradient. We distinguish two cases.  
First, assume that $\rho^u>0$. Then, the optimality conditions read 
$$
\begin{pmatrix}
    -\lambda \frac{\abs{j^v}^2 + \abs{g^v}^2}{2\rho^v} + \rho^v - \rho^u \\
    \lambda \frac{j^v}{\rho^v} + j^v - j^u\\
    \lambda \frac{g^v}{\rho^v} + g^v - g^u
\end{pmatrix} = 0.
$$
Similar to \cite[Appendix A]{Pietschmann_2022}, direct calculations verify that this system is solved by $\rho^v = \rho^*$, $j^v = \frac{\rho^v j^u}{\rho^v + \lambda }$ and $g^v = \frac{\rho^v g^u}{\rho^v + \lambda}$ if $\rho^*$ is the largest real root of the polynomial
$$
P(X) = (X-\rho^u)(X-\lambda)^2 - \frac{\lambda}{2}\left(\abs{j^u}^2 + \abs{g^u}^2\right).
$$
If $\rho^u=0$, the choice $v = (\rho^v, j^v, g^v) = (0,0,0)$ attains the minimum allowing us to efficiently evaluate the proximal mappings.

We conclude this section with two different examples for the algorithm introduced above.

\begin{example}[Counterexample to geodesic convexity]
    In Section~\ref{sec:discussion} we discussed the lack of geodesic convexity of the entropy on metric graphs and its connection to the $\Gamma$-convergence of the dynamic problems. An explicit counterexample to this convexity has been introduced in \cite[Section 4]{MatthiasErbar2022}, which we treat numerically here. We also compare the choices $\beta\in\{10^{-1},1\}$ in the dynamic Schrödinger problem to the explicitly known geodesic for $\beta=0$.
    To this end, consider the graph from Figure~\ref{fig:counterex_graph} with $\ell_\sfe=1$ and $n^\sfe_{\sfv_0}=1$ for all $\sfe\in\sfE$. Further, let
    \begin{align}\label{eq:indicator}
        \rho_0(x) = \begin{cases}
            \frac{1}{\ve}\one_{[0,\ve]}(x) &: x\in \sfe_1 \text{ or } x \in \sfe_2\\ 0&: \text{else}
        \end{cases} \quad \text{and}\quad \rho_1(x) = \begin{cases}
            \frac{1}{2\ve}\one_{[0,\ve]}(x) &: x\in \sfe_3\\ 0&: \text{else}
        \end{cases} 
    \end{align}
    for $0<\ve<1$. We regularize the indicator functions by gaussian smoothing. Applying Algorithm~\ref{Alg1} to this problem yields the results shown in Figure~\ref{fig:counterex_dens}. As predicted by Theorem~\ref{thm:gamma_cvg_dynamic}, the minimizers approach the Wasserstein geodesic ($\beta=0$) as the regularization parameter tends to zero. In this regard, the numerical experiment is in good agreement with our analytical findings.
    Moreover, smoothing caused by the additional diffusive term can be observed as well as continuity in the vertices for $\beta>0$. For $\beta=0$ no such smoothing occurs and solutions are generally not continuous.

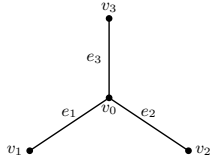
\begin{figure}[t]
    \centering
    \scalebox{.7}{
    \begin{tikzpicture}
        \draw[-] (-1.5,-1) -- (0,0) node[midway, above] {\small $e_1$};
	    \draw[-] (1.5,-1) -- (0,0) node[midway, above] {\small $e_2$};
	      \draw[-] (0,1.5) -- (0,0) node[midway, left] {\small $e_3$};
	
	      \filldraw (-1.5,-1) circle[radius=0.05] node[left] {\small $v_1$};
	      \filldraw (0,0) circle[radius=0.05] node[below] {\small $v_0$};
	      \filldraw (1.5,-1) circle[radius=0.05] node[right] {\small $v_2$};
	      \filldraw (0,1.5) circle[radius=0.05] node[above] {\small $v_3$};
    \end{tikzpicture}}
    \caption{Star-shaped graph with three edges}
    \label{fig:counterex_graph}
\end{figure}

\begin{figure}
    \centering
    \begin{subfigure}{.3\textwidth}
      \centering
      \includegraphics[width=\linewidth]{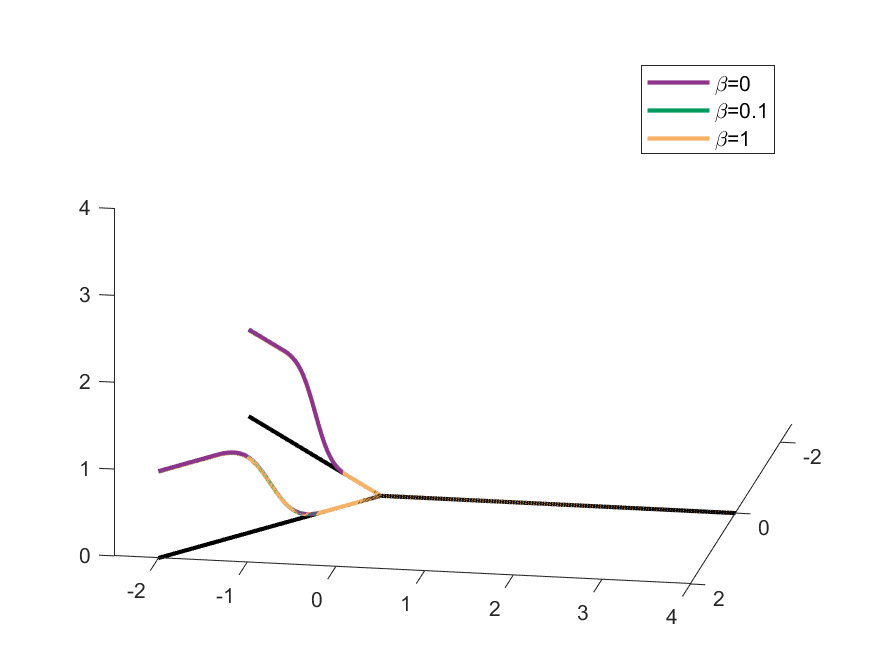}
      \caption{$\rho_t$ at $t=0$}
    \end{subfigure}
    \begin{subfigure}{.3\textwidth}
      \centering
      \includegraphics[width=\linewidth]{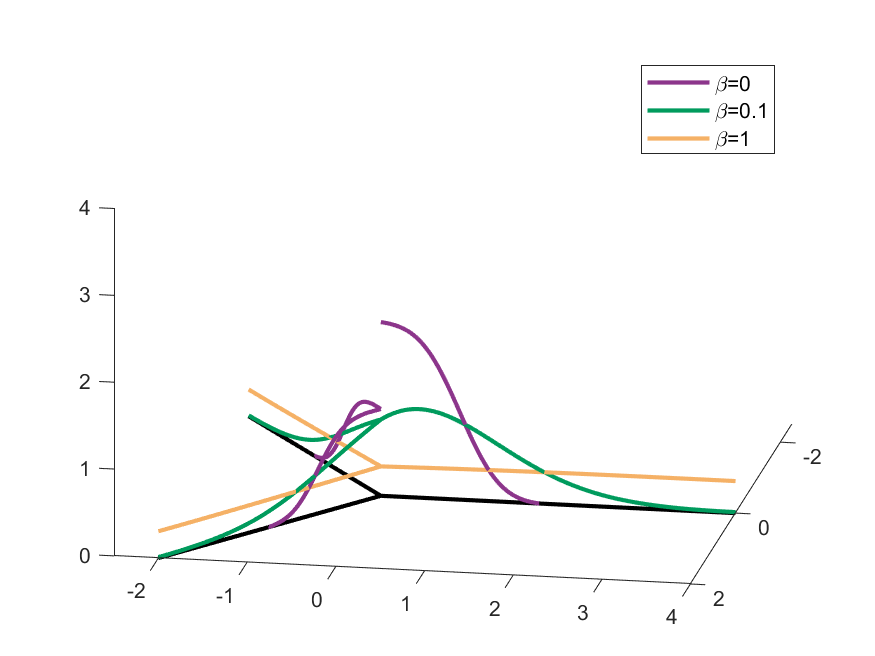}
      \caption{$\rho_t$ at $t=0.5$}
    \end{subfigure}
    \begin{subfigure}{.3\textwidth}
      \centering
      \includegraphics[width=\linewidth]{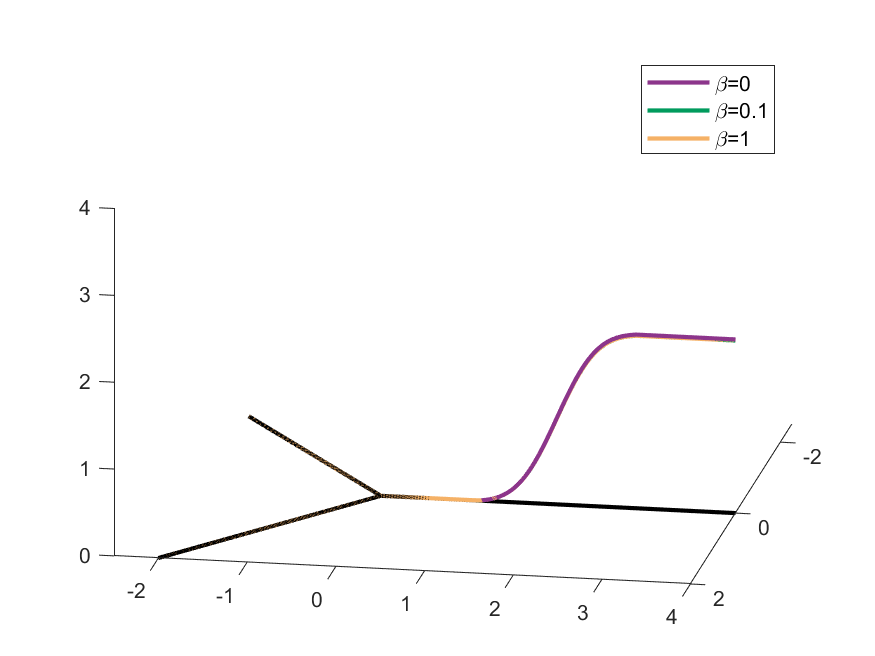}
      \caption{$\rho_t$ at $t=1$}
    \end{subfigure}
    \caption{Comparison between numerical results for $N_t=50$, $N_x=150$, $\delta_i=\delta=10^{-5}$, $\lambda=10^{-4}$, tolerance $2\cdot 10^{-3}$ and the choices $\beta\in\{10^{-1},1\}$ with the explicit geodesic for $\beta=0$. Initial and final data are defined according to~\eqref{eq:indicator} with $\ve=\frac12$.}
    \label{fig:counterex_dens}
\end{figure}
    
\end{example}

\begin{example}[Gaussian data]\label{ex:gaussian}
    In this example, we consider the same metric graph as displayed in Figure~\ref{fig:counterex_graph}, while imposing gaussian initial and final data supported on one edge each, i.e.
    \begin{align}\label{eq:gaussian}
        \rho_0(x) = \begin{cases}
             C\exp\bra*{-\bra*{\frac{x-\Bar{x}}{\sfs}}^2} &: x\in \sfe_1\\ 0&: \text{else}
        \end{cases} \quad\text{and}\quad \rho_1(x) = \begin{cases}
            \exp\bra*{-\bra*{\frac{x-\Bar{x}}{\sfs}}^2} &: x\in \sfe_3\\ 0&: \text{else}
        \end{cases}
    \end{align}
    for $\Bar{x}=\frac{1}{2}$, $\sfs=\frac{1}{10}$, and a constant $C>0$ such that the resulting measures satisfy $\mu_i(\omgg)=1$, $i=0,1$. We consider different parameters $\beta\in\{10^{-1},1\}$ and compare the curves with the explicit geodesic for $\beta=0$. The numerical results are presented in Figure~\ref{fig:gaussian_data} and again show convergence as $\beta\to0$. Additionally, in the case $\beta>0$ we observe smoothing.
    For dynamic transport, i.e. $\beta=0$, it is also shown that the solution is supported only on the edges $\sfe_1$ and $\sfe_3$, thus behaving like a one-dimensional Wasserstein geodesic. In particular, the restriction of this curve to the edge $\sfe_2$ is constant and equal to zero. 
    In contrast, the results for $\beta>0$ show support on the whole graph as well as continuity in the vertices similar to the previous example. 
    
    \begin{figure}
    \centering
    \begin{subfigure}{.3\textwidth}
      \centering
      \includegraphics[width=\linewidth]{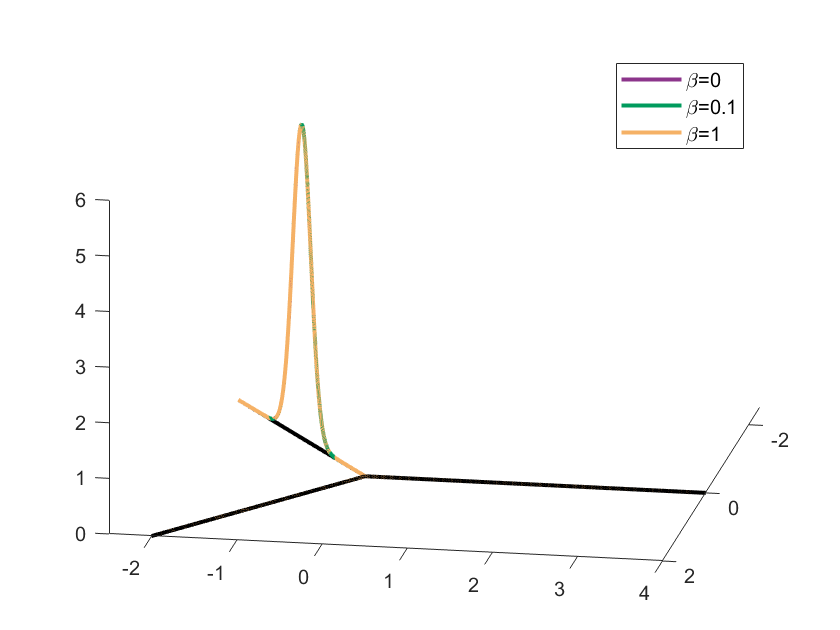}
      \caption{$\rho_t$ at $t=0$}
    \end{subfigure}
    \begin{subfigure}{.3\textwidth}
      \centering
      \includegraphics[width=\linewidth]{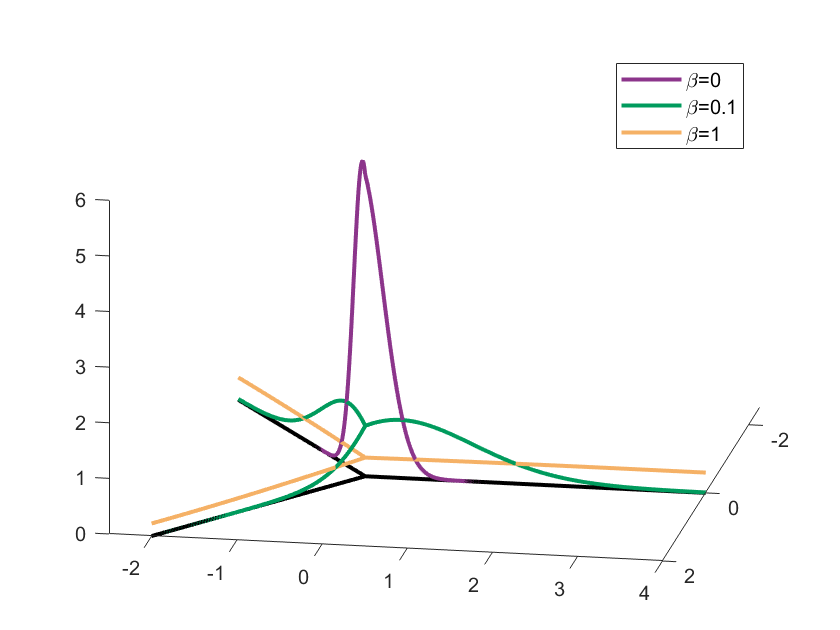}
      \caption{$\rho_t$ at $t=0.5$}
    \end{subfigure}
    \begin{subfigure}{.3\textwidth}
      \centering
      \includegraphics[width=\linewidth]{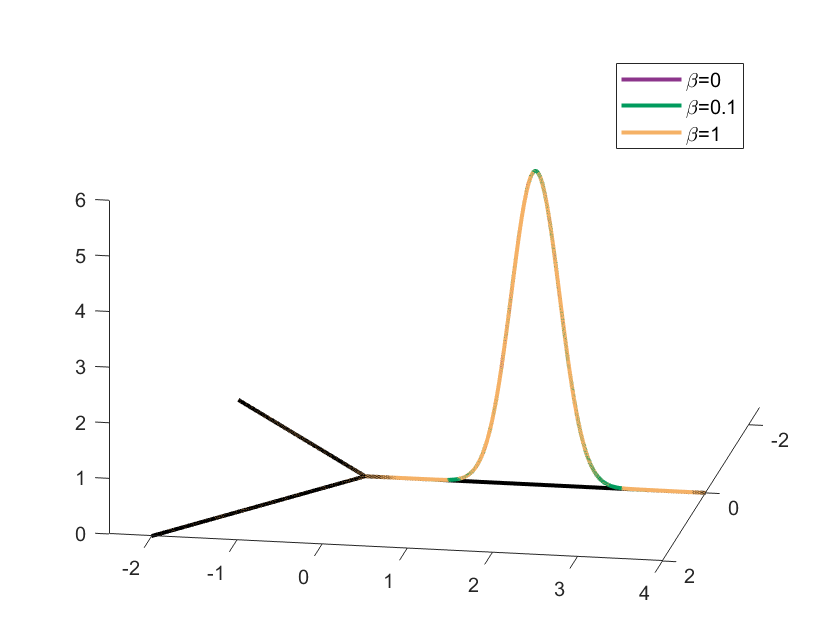}
      \caption{$\rho_t$ at $t=1$}
    \end{subfigure}
    \caption{Comparison between numerical results in the setting of Example~\ref{ex:gaussian} for $N_t=50$, $N_x=150$, $\delta_i=\delta=10^{-5}$, $\lambda=10^{-4}$, tolerance $2\cdot 10^{-3}$ and the choices $\beta\in\{10^{-1},1\}$ with the explicit geodesic for $\beta=0$. }
    \label{fig:gaussian_data}
\end{figure}
\end{example}

\section*{Acknowledgements}
The authors thank Delio Mugnolo (FernUniversität in Hagen – University of Hagen) for fruitful discussions on properties of the heat kernel. This article is based upon work from COST Action 24122 mSPACE, supported by COST (European Cooperation in Science and Technology), www.cost.eu.

\nocite*
\bibliographystyle{alpha}
\bibliography{sample}

\end{document}